\documentclass[a4paper,12pt]{amsart}

\usepackage[margin=4cm]{geometry}
\usepackage[T1]{fontenc}

\usepackage{amsthm, amsmath, amssymb, bbm, bm}
\usepackage{mathrsfs}

\usepackage{accents}

\usepackage[dvipsnames]{xcolor}

\usepackage{tikz}
\usetikzlibrary{calc}
\usetikzlibrary{intersections}
\usepackage{tikz-3dplot}

\usepackage{enumerate}
\usepackage{graphics}

\usepackage[all]{xy}
\makeatother
\newdir{ >}{{}*!/-10pt/@{>}}
\makeatletter

\usepackage[colorlinks=true, pdfstartview=FitV, linkcolor=blue, %
citecolor=blue, urlcolor=blue]{hyperref}

\usepackage{marginnote}
\usepackage{xspace}
\usepackage{ulem}
\newcommand{\nc}{\newcommand}
\newenvironment{rouge}
{\relax\color{red}}
{\hspace*{.3ex}\relax}
\newcommand{\ber}{\begin{rouge}{}\marginnote{\mbox{$\bullet$}}{}}
\newcommand{\er}{\end{rouge}}
\newcommand{\bera}{\begin{rouge}{}\marginnote{\fbox{\scshape\lowercase{A}}}{}}
\newcommand{\berm}{\begin{rouge}{}\marginnote{\fbox{\scshape\lowercase{M}}}{}}

\newcommand{\erm}{\end{rouge}}

\newenvironment{bleu}
{\relax\color{blue}}
{\hspace*{.3ex}\relax}
\newcommand{\beb}{\begin{bleu}}
\newcommand{\bebm}{\begin{bleu}{}\marginnote{\fbox{\scshape\lowercase{M}}}{}}
\newcommand{\beba}{\begin{bleu}{}\marginnote{\fbox{\scshape\lowercase{A}}}{}}
\newcommand{\eb}{\end{bleu}}
\newenvironment{blanc}
{\relax\color{YellowOrange}}
{\hspace*{.3ex}\relax}

\newcommand{\bw}{\begin{blanc}{}\marginnote{\mbox{$\bullet$}}{}}
\newcommand{\ew}{\end{blanc}}

\usepackage{units}

\usepackage{accents}

\usepackage{tensor}

\theoremstyle{plain}

\newtheorem{theorem}{Theorem}[section]
\newtheorem*{theorem*}{Theorem}
\newtheorem{corollary}[theorem]{Corollary}
\newtheorem{proposition}[theorem]{Proposition}
\newtheorem{lemma}[theorem]{Lemma}

\theoremstyle{definition}
\newtheorem{definition}[theorem]{Definition}

\newtheorem{example}[theorem]{Example}
\newtheorem*{example*}{Example}
\newtheorem{notation}[theorem]{Notation}

\nc{\Lemma}{\begin{lemma}}
\nc{\enlemma}{\end{lemma}}
\nc{\Prop}{\begin{proposition}}
\nc{\enprop}{\end{proposition}}
\nc{\Def}{\begin{definition}}
\nc{\edf}{\end{definition}}

\nc{\scup}{\mathop{\scalebox{.8}{$\displaystyle\bigcup$}}\mspace{1mu}\limits}
\nc{\scap}{\mathop{\scalebox{.8}{$\displaystyle\bigcap$}}\limits}
\nc{\ssqcup}{\mathop{\scalebox{.8}{$\displaystyle\bigsqcup$}}\limits}

\newcommand{\union}{\cup}
\newcommand{\Union}{\bigcup\limits}

\newcommand{\C}{\mathbb{C}}

\newcommand{\R}{\mathbb{R}}

\newcommand{\Z}{\mathbb{Z}}

\DeclareMathOperator{\id}{id}

\newcommand{\derived}[1]{\mathrm{#1}}

\newcommand{\derd}{\derived{D}}
\newcommand{\dere}{\derived{E}}

\newcommand{\derr}{\derived{R}}
\newcommand{\derl}{\derived{L}}
\nc{\derb}{\derd^{\mathrm{b}}}
\newcommand{\BDC}{\derd^{\mathrm{b}}}

\nc{\soplus}{\scalebox{.65}{\raisebox{.2ex}{$\displaystyle\bigoplus$}}}

\newcommand{\dsum}[1][]{\mathbin{\oplus_{#1}}}

\newcommand{\ilim}[1][]{\mathop{\varinjlim}\limits_{#1}}

\renewcommand{\to}[1][]{\xrightarrow{#1}}
\newcommand{\from}[1][]{\xleftarrow{#1}}
\newcommand{\isofrom}[1][]{\xleftarrow[#1]%
{\raisebox{-.4ex}[0ex][-.4ex]{$\mspace{2mu}\sim\mspace{2mu}$}}}
\newcommand{\isoto}[1][]{\xrightarrow[#1]{%
{\raisebox{-.6ex}[0ex][0ex]{$\mspace{1mu}\sim\mspace{2mu}$}}}}

\newcommand{\Endo}[1][]{\mathrm{End}_{\raise1.5ex\hbox to.1em{}#1}}

\newcommand{\Hom}[1][]{\mathrm{Hom}_{\raise1.5ex\hbox to.1em{}#1}}

\newcommand{\RHom}[1][]{\derr\mathrm{Hom}_{\raise1.5ex\hbox to.1em{}#1}}

\newcommand{\Ext}[2][]{\mathrm{Ext}_{\raise1.5ex\hbox to.1em{}#1}^{#2}}

\newcommand{\Tens}[1][]{\mathbin{\otimes_{\raise1.5ex\hbox to-.1em{}#1}}}

\newcommand{\LTens}[1][]{\mathbin{\otimes_{\raise1.5ex\hbox to-.1em{}#1}^{\derl}}}

\newcommand{\Tor}[2][]{\mathrm{Tor}^{\raise1.5ex\hbox to.1em{}#1}_{#2}}

\newcommand{\sheaffont}[1]{\mathcal{#1}}

\def\shi{\sheaffont{I}}

\def\shm{\sheaffont{M}}

\def\shp{\sheaffont{P}}

\newcommand{\rsect}{\derr\varGamma}

\newcommand{\shendo}[1][]{{\sheaffont{E}nd}_{\raise1.5ex\hbox to.1em{}#1}}

\renewcommand{\hom}[1][]{{\sheaffont{H}om}_{\raise1.5ex\hbox to.1em{}#1}}
\newcommand{\aut}[1][]{{\sheaffont{A}ut}_{\raise1.5ex\hbox to.1em{}#1}}
\newcommand{\inn}[1][]{{\sheaffont{I}nn}_{\raise1.5ex\hbox to.1em{}#1}}

\newcommand{\rhom}[1][]{{\derr\sheaffont{H}om}_{\raise1.5ex\hbox to.1em{}#1}}

\newcommand{\ext}[2][]{{\sheaffont{E}xt}_{\raise1.5ex\hbox to.1em{}#1}^{#2}}

\newcommand{\thom}[1][]{{\sheaffont{T}hom}_{\raise1.5ex\hbox to.1em{}#1}}

\newcommand{\tens}[1][]{\mathbin{\otimes_{\raise1.5ex\hbox to-.1em{}#1}}}

\newcommand{\ltens}[1][]{\mathbin{\otimes_{\raise1.5ex\hbox to-.1em{}#1}^{\derl}}}

\newcommand{\tor}[2][]{{\sheaffont{T}or}^{\raise1.5ex\hbox to.1em{}#1}_{#2}}

\newcommand{\etens}[1][]{\mathbin{\boxtimes_{\raise1.5ex\hbox to-.1em{}#1}}}

\newcommand{\oim}[1]{#1_*}
\newcommand{\eim}[1]{#1_!}

\newcommand{\roim}[1]{\derr#1_*}

\newcommand{\reim}[1]{\derr#1_{\mspace{.5mu}!}\mspace{2mu}}

\newcommand{\reeim}[1]{\derr#1_{\mspace{1mu}!!}\mspace{1mu}}

\newcommand{\opb}[1]{#1^{-1}}

\newcommand{\epb}[1]{#1^{\mspace{1.5mu}!}\mspace{2mu}}

\newcommand{\tenstop}[1][]{\mathbin{\hat{\otimes}_{\raise1.5ex\hbox to-.1em{}#1}}}

\newcommand{\homtop}[1][]{\sheaffont{L}_{\raise1.5ex\hbox to.1em{}#1}}

\newcommand{\Homtop}[1][]{\mathrm{L}_{\raise1.5ex\hbox to.1em{}#1}}

\newcommand{\D}{\sheaffont{D}}

\renewcommand{\O}{\sheaffont{O}}

\newcommand{\detens}[1][]%
{\mathbin{\boxtimes_{\raise1.5ex\hbox to-.1em{}#1}^{\mspace{2mu}\mathsf{D}}}}

\newcommand{\dtens}[1][]{\mathbin{\otimes_{\raise1.5ex\hbox to-.1em{}#1}^{\mathsf{D}}}}

\newcommand{\hol}{\mathrm{hol}}

\newcommand{\reghol}{{\mathrm{rh}}}

\renewcommand{\leq}{\leqslant}
\renewcommand{\geq}{\geqslant}

\newcommand{\field}{\mathbf{k}}
\newcommand{\ind}{\mathrm{I}\mspace{2mu}}
\newcommand{\ifield}{\ind\field}

\newcommand{\Rc}{{\R\text-\mathrm{c}}}
\newcommand{\Cc}{{\C\text-\mathrm{c}}}

\newcommand{\reg}{{\operatorname{reg}}}
\newcommand{\stab}{{\operatorname{st}}}

\newcommand{\cl}{\colon}

\newcommand{\ctens}{\mathbin{\mathop\otimes\limits^+}}
\newcommand{\cetens}{\mathbin{\mathop\boxtimes\limits^+}}
\newcommand{\cihom}{{\derr\shi hom}^+}

\newcommand{\dr}{\mathcal{DR}}

\renewcommand{\Re}{\operatorname{Re}}

\newcommand{\ihom}[1][]{{\shi hom}_{\raise1.5ex\hbox to.1em{}#1}}
\newcommand{\rihom}[1][]{{\derr\mspace{2mu}\shi hom}_{\raise1.5ex\hbox to.1em{}#1}}
\newcommand{\ii}[1][]{{\sheaffont{I}h}_{\raise1.5ex\hbox to.1em{}#1}}

\newcommand{\indlim}[1][]{\mathop{\text{\rm``$\varinjlim$''}}\limits_{#1}}

\newcommand{\dcomp}[1][]{\mathbin{\circ_{\raise1.5ex\hbox to-.1em{}#1}^{\mathsf{D}}}}

\newcommand{\enh}{\derived{E}}

\newcommand{\drE}[1][X]{\mathcal{DR}^\enh_{#1}}

\newcommand{\fhom}{\rhom^\enh}
\newcommand{\FHom}{\RHom^\enh}
\newcommand{\fihom}{\rihom^\enh}

\newcommand{\BEC}[2][\ifield]{\dere^{\mathrm{b}}(#1_{#2})}
\newcommand{\BECRc}[2][\ifield]{\dere^{\mathrm{b}}_\Rc(#1_{#2})}
\newcommand{\BECwRc}[2][\ifield]{\dere^{\mathrm{b}}_\wRc(#1_{#2})}
\newcommand{\BECcon}[2][\ifield]{\dere^{\mathrm{b}}_{\bGmp}(#1_{#2})}
\newcommand{\BECp}[2][\ifield]{\dere^{\mathrm{b}}_+(#1_{#2})}
\newcommand{\BECs}[2][\ifield]{\dere^{\mathrm{b}}_\stab(#1_{#2})}

\newcommand{\BECm}[2][\ifield]{\dere^{\mathrm{b}}_-(#1_{#2})}
\newcommand{\BECpm}[2][\ifield]{\dere^{\mathrm{b}}_\pm(#1_{#2})}

\newcommand{\Edual}{\dual^\enh}
 \newcommand{\Eoim}[1]{\enh#1_*}

\newcommand{\Eeeim}[1]{\enh#1_{!!}}

\newcommand{\Eopb}[1]{\enh#1^{-1}}
\newcommand{\Eepb}[1]{\enh\mspace{1mu}#1^{\mspace{1.5mu}!}}

\newcommand{\LE}{\operatorname{L^\enh}}
\newcommand{\RE}{\operatorname{R^\enh}}

\newcommand{\semicolon}{\nobreak \mskip2mu\mathpunct{}\nonscript\mkern-\thinmuskip{;}\mskip6mu plus1mu\relax}

\newcommand{\dual}{\mathrm{D}}

\newcommand{\defeq}{\mathbin{:=}}
\newcommand{\eqdef}{\mathbin{=:}}
\newcommand{\bl}{\bigl(}
\newcommand{\br}{\bigr)}
\newcommand{\To}[1][]{\xrightarrow[]{\mspace{10mu}{#1}\mspace{10mu}}}

\newenvironment{myarray}[1]{\relax\setlength{\arraycolsep}{1pt}

\begin{array}{#1}}{\end{array}\relax}

\newcommand{\ba}{\begin{myarray}}
\newcommand{\ea}{\end{myarray}}

\newcommand{\be}{\begin{enumerate}}
\newcommand{\ee}{\end{enumerate}}
\newcommand{\bnum}{\be[{\rm(i)}]}

\nc{\bwr}{\mbox{\large{$\wr$}}}
\nc{\vphi}{\varphi}
\nc{\seteq}{\mathbin{:=}}
\nc{\noi}{\noindent}
\nc{\ro}{{\rm(}}
\nc{\rf}{{\rm)}\xspace}
\nc{\ms}{\mspace}
\nc{\sbcup}{\mathop{\scalebox{0.75}{$\displaystyle\bigcup$}}}
\nc{\ol}{\overline}
\nc{\scbul}{{\,\raise1pt\hbox{$\scriptscriptstyle\bullet$}\,}}
\nc{\set}[2]{\left\{#1\;\semicolon\; #2 \right\}}
\nc{\extp}{\mathop{\raisebox{.3ex}{\scalebox{0.8}{$\displaystyle\bigwedge$}}}\limits}

\newenvironment{myequation}
{\relax\setlength{\arraycolsep}{1pt}\begin{eqnarray}}
{\end{eqnarray}}
\newenvironment{myequationn}
{\relax\setlength{\arraycolsep}{1pt}\begin{eqnarray*}}
{\end{eqnarray*}}

\newenvironment{myalign}
{\relax\begin{align}}
{\end{align}}
\newenvironment{myalignn}
{\relax\begin{align*}}
{\relax\end{align*}}

\nc{\eq}{\begin{myequation}}
\nc{\eneq}{\end{myequation}}
\nc{\eqn}{\begin{myequationn}}
\nc{\eneqn}{\end{myequationn}}

\nc{\eqa}{\begin{myalign}}
\nc{\eneqa}{\end{myalign}}
\nc{\eqan}{\begin{myalignn}}
\nc{\eneqan}{\end{myalignn}}

\nc{\on}{\operatorname}
\nc{\Ind}{\on{Ind}}
\nc{\Proof}{\begin{proof}}
\nc{\QED}{\end{proof}}
\nc{\cor}{\field}
\nc{\tone}{\To[+1]}
\renewcommand{\ge}{\geq}
\renewcommand{\le}{\leq}

\newcommand{\LEp}{\operatorname{L}^\enh_+}

\newcommand{\REp}{\operatorname{R}^\enh_+}

\renewcommand{\tor}{\mathrm{tor}}

\newcommand{\bclose}[1]{{\accentset{\vee}{#1}}}
\newcommand{\unbordered}[1]{{\accentset{\circ}{#1}}}

\newcommand{\bR}{{\R_\infty}}

\newcommand{\cR}{{\overline\R}}

\newcommand{\bM}{\inbordered{M}}
\nc{\unb}{\unbordered}
\nc{\eps}{\varepsilon}
\nc{\inb}{\inbordered}
\nc{\colim}{\varinjlim\limits}
\nc{\ssubset}{\subset\ms{-3mu}\subset}
\nc{\al}{\alpha}
\nc{\qtq}[1][and]{\quad\text{#1}\quad}
\nc{\qt}[1]{\quad\text{#1}}
\nc{\olG}[1][f]{{\overset{\ms{4mu}\rule[-.05ex]{1.6ex}{.115ex}}{\Gamma}}_{%
\ms{-3mu}#1}}

\newcommand{\bN}{\inbordered{N}}

\nc{\cf}{\bclose{f}}
\nc{\cp}{\bclose{p}}

\newcommand{\inbordered}[1]{{#1_\infty}}

\newcommand{\bZ}{\inbordered{Z}}

\newcommand{\bU}{\inbordered{U}}

\DeclareMathOperator{\quot}{Q}

\newcommand{\Efield}{\field^\enh}

\newcommand{\Ex}{\mathbb{E}}
\newcommand{\ex}{\mathsf{E}}
\newcommand{\exx}{\ex}

\newcommand{\Lap}{\mathsf{L}}
\newcommand{\lap}{{}^\Lap}
\newcommand{\Lapa}{{\mathpalette\Lapatemp\relax}}
\newcommand{\Lapatemp}[2]{\reflectbox{$#1\Lap$}}
\newcommand{\lapa}{{}^{\Lapa}}

\nc{\tM}{\widetilde{M}}
\nc{\tX}{\widetilde{X}}
\nc{\ti}{{\tilde\imath}}
\nc{\tj}{{\tilde\jmath}}

\newcommand{\dt}[1]{\widetilde{#1}^{\mathsf d}}
\newcommand{\dtd}[1]{\widetilde{#1}^{\mathsf d,\times}}
\nc{\dtM}{\dt M}
\nc{\dtdM}{\dtd M}
\nc{\dti}{\tilde\imath^{\mathsf d}}

\newcommand{\range}[2]{#1 \mathbin{\rhd} #2}

\newcommand{\dotowns}{\mathbin{\bdot\owns}}
\newcommand{\dotsupset}{\mathbin{\bdot\supset}}

\newcommand{\wRc}{{\mathrm{w}\text-\R\text-\mathrm{c}}}

\newcommand{\rb}{\mathsf{rb}}
\newcommand{\pb}{\mathsf{pb}}
\newcommand{\nd}{\mathsf{nd}}
\newcommand{\sm}{\mathsf{sm}}
\newcommand{\sph}{\mathsf{sph}}
\nc{\st}[1]{{\{{#1}\}}}
\nc{\bP}{\mathbb{P}}
\nc{\into}{\hookrightarrow}
\nc{\fR}{{\R_\infty}}
\nc{\cS}{\bclose{S}}
\nc{\RB}[2][N]{#2_{#1}^{\rb}}
\nc{\prb}{p_{\rb}}
\nc{\PB}[2][N]{#2_{#1}^{\pb}}
\nc{\ppb}{p_{\pb}}
\nc{\ND}[2][N]{#2_{#1}^{\nd}}
\nc{\pnd}{p_{\nd}}
\nc{\snd}{s_{\nd}}
\nc{\dND}[2][N]{\bdot{#2}_{#1}^\nd}
\nc{\psm}{p_{\sm}}

\newcommand{\sh}{\mathsf{sh}}
\newcommand{\Ish}{\mathrm{I}\mathsf{sh}}

\newcommand{\Enu}{\enh\nu}
\newcommand{\Emu}{\enh\mu}

\newcommand*\bigcdot{\mathpalette\bigcdot@{.5}}
\newcommand*\bigcdot@[2]{\mathbin{\vcenter{\hbox{\scalebox{#2}{$\m@th#1\bullet$}}}}}
\newcommand{\bdot}[1]{{\accentset{\bigcdot}{#1}}\vphantom{#1}}

\newcommand{\Gmp}{\R^\times_{>0}}
\newcommand{\bGmp}{\inb{(\Gmp)}}

\newcommand{\Bd}{\mathsf{B}}

\newcommand{\bbM}{\mathsf{M}}
\newcommand{\ccM}{C}
\newcommand{\obM}{\unbordered{\bbM}}
\newcommand{\cbM}{\bclose{\bbM}}

\newcommand{\bbN}{\mathsf{N}}
\newcommand{\obN}{\unbordered{\bbN}}
\newcommand{\cbN}{\bclose{\bbN}}

\newcommand{\bbS}{\mathsf{S}}

\newcommand{\cc}{\bclose}
\newcommand{\oo}{\unbordered}
\newcommand{\of}{\oo f}
\nc{\oloG}[1][\of]{{\overset{\ms{4mu}\rule[-.05ex]{1.6ex}{.115ex}}{\Gamma}}_{%
\ms{-3mu}#1}}

\nc{\inc}[1][\bbM]{k_{#1}}

\nc{\ct}{\ol{t}}

\usepackage[mode=image]{standalone} 

\begin{document}

\title[On a topological counterpart of regularization]{On a topological counterpart of regularization for holonomic $\D$-modules}

\author[A.~D'Agnolo]{Andrea D'Agnolo}
\address[Andrea D'Agnolo]{Dipartimento di Matematica\\
Universit{\`a} di Padova\\
via Trieste 63, 35121 Padova, Italy}
\thanks{The research of A.D'A.\
was partially supported by GNAMPA/INdAM. He acknowledges the kind hospitality at RIMS of
Kyoto University during the preparation of this paper.
}
\email{dagnolo@math.unipd.it}

\author[M.~Kashiwara]{Masaki Kashiwara}
\thanks{The research of M.K.\
was supported by Grant-in-Aid for Scientific Research (B)
15H03608, Japan Society for the Promotion of Science}
\address[Masaki Kashiwara]{
Kyoto University Institute for Advanced study,
Research Institute for Mathematical Sciences, Kyoto University,
Kyoto 606-8502, Japan \& Korea Institute for Advanced Study, Seoul 02455, Korea}
\email{masaki@kurims.kyoto-u.ac.jp}


\maketitle

\begin{abstract}
On a complex manifold, the embedding of the category of regular holonomic $\D$-modules into that of holonomic $\D$-modules has a left quasi-inverse functor $\shm\mapsto\shm_\reg$, called regularization. Recall that $\shm_\reg$ is reconstructed from the de Rham complex of $\shm$ by the regular Riemann-Hilbert correspondence. Similarly, on a topological space, the embedding of sheaves into enhanced ind-sheaves has a left quasi-inverse functor, called here sheafification. Regularization and sheafification are intertwined by the irregular Riemann-Hilbert correspondence.
Here, we study some of the properties of the sheafification functor. In particular, we provide a germ formula for the sheafification of enhanced specialization and microlocalization.
\end{abstract}

\tableofcontents

\addtocontents{toc}{\protect\setcounter{tocdepth}{1}}
\numberwithin{equation}{section}

\section{Introduction}
Let $X$ be a complex manifold. 
The regular Riemann-Hilbert correspondence (see \cite{Kas84}) states that 
the de Rham functor induces an equivalence between the triangulated category of regular holonomic $\D$-modules and that of $\C$-constructible sheaves. More precisely, one has a diagram 
\begin{equation}\label{eq:introRH}
\xymatrix@R=3ex@C=3em{
\BDC_\hol(\D_X) \ar@<1ex>[dr]^{\dr} \\
\BDC_\reghol(\D_X) \ar@{ >->}[u]^\iota \ar@{ ->}@<.6ex>[r]^-\dr_-\sim  
& \BDC_\Cc(\C_X) \ar@{ ->}@<1ex>[l]^{\Phi}
}
\end{equation}
where $\iota$ is the embedding (i.e.\ fully faithful functor) of regular holonomic $\D$-modules into holonomic $\D$-modules, the triangle quasi-commutes, $\dr$ is the de Rham functor, and $\Phi$ is an (explicit) quasi-inverse to $\dr$. 

The {\em regularization} functor $\reg\colon\BDC_\hol(\D_X)\to\BDC_\reghol(\D_X)$ is defined by $\shm_\reg\defeq \Phi(\dr(\shm))$. It is a left quasi-inverse to $\iota$, of transcendental nature.
Recall that $(\iota,\reg)$ is \emph{not} a pair of adjoint functors\footnote{By saying that $(\iota,\reg)$ is a pair  of adjoint functors, we mean that $\iota$ is the left adjoint of $\reg$.}. Recall also that $\reg$ is conservative\footnote{In fact, if $\shm_\reg\simeq 0$ then $\dr(\shm)\simeq\dr(\shm_\reg)\simeq 0$, and hence $\shm\simeq0$.}.

Let $\field$ be a field and $M$ be a good topological space. Consider the natural embeddings $\xymatrix@C=3ex{\BDC(\field_M) \ar@{ >->}[r]^\iota & \BDC(\ifield_M) \ar@{ >->}[r]^e & \BECs M }$ of sheaves into ind-sheaves into stable enhanced ind-sheaves. 
One has pairs of adjoint functors $(\alpha,\iota)$ and $(e,\Ish)$, and we set $\sh\defeq\alpha\,\Ish$:
\[
\sh\colon \BECs M \to[\Ish] \BDC(\ifield_M) \to[\alpha] \BDC(\field_M).
\]
We call $\Ish$ and $\sh$ the {\em ind-sheafification} and {\em sheafification} functor, respectively. The functor $\sh$ is a left quasi-inverse of $e\,\iota$. 

For $\field=\C$ and $M=X$, the irregular Riemann-Hilbert correspondence (see \cite{DK16}) 
intertwines the pair $(\iota,\reg)$ with the pair $(e\,\iota,\sh)$.
In particular, the pair $(e\,\iota,\sh)$ is \emph{not} a pair of adjoint functors in general.

With the aim of better understanding the rather elusive regularization functor,
in this paper we study some of the properties of the ind-sheafification and sheafification functors.

More precisely, the contents of the paper are as follows.

In \S\ref{se:not}, besides recalling notations, we establish some complementary results on ind-sheaves on bordered spaces that we need in the following. 
Further complements are provided in  Appendix~\ref{se:compl}.

Some functorial properties of ind-sheafification and sheafification are obtained in \S\ref{se:sh}. 
In \S\ref{se:germ}, we obtain a germ formula for the sheafification of a pull-back by an embedding. Then, these results are used in section \S\ref{se:numu} to obtain a germ formula for the sheafification of enhanced specialization and microlocalization. In particular, the formula for the specialization puts in a more geometric perspective what we called \emph{multiplicity test functor} in \cite[\S6.3]{DK18}. 

Finally, we provide in Appendix~\ref{se:wc} a formula for the sections of a weakly constructible sheaf on a locally closed subanalytic subset, which could be of independent interest.

\section{Notations and complements}\label{se:not}
We recall here some notions and results, mainly to fix notations, referring to the literature for details. In particular, we refer to \cite{KS90} for sheaves, to \cite{Tam08} (see also \cite{GS14,DK19}) for enhanced sheaves, to \cite{KS01} for ind-sheaves, and to \cite{DK16} (see also \cite{KS16D,Kas16,DK19}) for bordered spaces and enhanced ind-sheaves. We also add some complements.

\medskip

In this paper, $\field$ denotes a base field.

A good space is a topological space which is Hausdorff, locally compact, countable at infinity, and with finite soft dimension. 
 
By subanalytic space we mean a subanalytic space which is also a good space.

\subsection{Bordered spaces}
The category of bordered spaces has for objects the pairs $\bbM=(M,\ccM)$ with $M$ an open subset of a good space $\ccM$. Set $\obM\defeq M$ and $\cbM\defeq\ccM$.
A morphism $f\colon\bbM\to\bbN$ is a morphism $\of\colon \obM\to \obN$ of good spaces such that the projection $\oloG\to \cbM$ is proper. Here, $\oloG$ denotes the closure in $\cbM\times\cbN$ of the graph $\Gamma_{\oo f}$ of $\of$.

Note that $\bbM\mapsto\cbM$ is not a functor.
The functor $\bbM\mapsto\obM$ is right adjoint to the embedding $M\mapsto(M,M)$ of good spaces into bordered spaces.
We will write for short $M=(M,M)$.

Note that the inclusion $\inc\colon\obM\to\cbM$ factors into
\begin{equation}
\label{eq:k}
\xymatrix{
\inc\;\colon\; \obM \ar[r]^-{i_\bbM} & \bbM\ar[r]^-{j_\bbM} &\cbM.
}
\end{equation}

By definition, a subset $Z$ of $\bbM$ is a subset of $\obM$.
We say that $Z\subset\bbM$ is open (resp.  closed, locally closed) if it is so in $\obM$.
For a locally closed subset $Z$ of $\bbM$,
we set $\bZ=(Z,\overline Z)$ where $\overline Z$ is the closure of $Z$ in $\cbM$.
Note that $\bU\simeq(U,\cbM)$ for $U\subset\bbM$ open.

We say that $Z$ is a relatively compact subset of $\bbM$
if it is contained in a compact subset of $\cbM$.
Note that this notion does not depend on the choice of $\cbM$. 

An {\em open covering} $\st{ U_i}_{i\in I}$ of a bordered space $\bbM$ is
an open covering of $\obM$ which satisfies the condition:
for any relatively compact subset $Z$ of $\bbM$ there exists a finite subset $I'$ of $I$ such that $Z\subset\sbcup_{i\in I'}U_i$.

We say that a morphism $f\colon\bbM\to\bbN$ is
\bnum
\item {\em an open embedding} if
$\oo f$ is a homeomorphism from $\obM$ onto an open subset of $\obN$,
\item
{\em borderly submersive} if 
there exists an open covering $\st{U_i}_{i\in I}$ of $\bbM$ 
such that for any $i\in I$ there exist
 a subanalytic space $S_i$ and an open embedding  
$g_i\cl\inb{(U_i)}\to S_i\times\bbN$
with
a commutative diagram of bordered spaces
$$\xymatrix@C=8ex@R=4ex{\inb{(U_i)}\ar[r]\ar[d]_{g_i}&\bbM\ar[d]^f\\
S_i\times\bbN\ar[r]_--{p_i}&\bbN,}$$
where $p_i$ is the projection,
\item
{\em semiproper} if $\oloG\to \cbN$ is proper,
\item
{\em proper} if it is semiproper and $\oo f\colon \obM\to \obN$ is proper,
\item
{\em self-cartesian} if the diagram $\xymatrix{
\obM \ar[r]_{\oo f} \ar[d]_{i_\bbM} & \obN \ar[d]^{i_\bbN} \\
\bbM \ar[r]^f & \bbN
}$ is cartesian.
\ee

Recall that, by \cite[Lemma~3.3.16]{DK16}, a morphism $f\colon\bbM\to\bbN$ is proper if and only if it is semiproper and self-cartesian.

\subsection{Ind-sheaves on good spaces}\label{sse:indingood}

Let $M$ be a good space.

We denote by $\BDC(\field_M)$ the bounded derived category of sheaves of $\field$-vector spaces on $M$.
For $S\subset M$ locally closed, we denote by $\field_S$ the extension by zero to $M$ of the constant sheaf on $S$ with stalk $\field$.

For $f\colon M\to N$ a morphism of good spaces, denote by $\tens$, $\opb f$, $\reim f$ and $\rhom$, $\roim f$, $\epb f$ the six operations. 
Denote by $\etens$ the exterior tensor and by $\dual_M$ the Verdier dual.

We denote by $\BDC(\ifield_M)$ the bounded derived category of ind-sheaves of $\field$-vector spaces on $M$, and by $\tens$, $\opb f$, $\reeim f$ and $\rihom$, $\roim f$, $\epb f$ the six operations. 
Denote by $\etens$ the exterior tensor and by $\dual_M$ the Verdier dual.

There is a natural embedding $\iota_M\colon\BDC(\field_M)\to\BDC(\ifield_M)$. It has a left adjoint $\alpha_M$, which in turn has a left adjoint $\beta_M$. 
The commutativity of these functors with the operations is as follows
\begin{equation}\label{eq:table}
\begin{tabular}{l||c|c|c|c|c}
{} & $\etens$ & $f^{-1}$ & $\roim f$ & $f^!$ & $\reeim f$  \\ \hline \hline 
$\underset{}{\iota}$  & $\circ$ & $\circ$ & $\circ$ & $\circ$ & $\times$ \\ \hline
$\underset{}{\alpha}$  & $\circ$ & $\circ$ & $\circ$ & $\times$ & $\circ$  \\ \hline
$\underset{}{\beta}$  & $\circ$ & $\circ$ & $\times$ & $\times$ & $\times$
\end{tabular}
\end{equation}
where ``$\circ$'' means that the functors commute,
and ``$\times$'' that they don't.

\subsection{Ind-sheaves on bordered spaces}

Let $\bbM$ be a bordered space.

Setting $\BDC(\field_\bbM)\defeq \BDC(\field_\cbM)/\BDC(\field_{\cbM\setminus\obM})$, one has $\BDC(\field_\bbM)\simeq\BDC(\field_\obM)$.

The bounded derived category of ind-sheaves of $\field$-vector spaces on $\bbM$ is defined by $\BDC(\ifield_\bbM)\defeq \BDC(\ifield_\cbM)/\BDC(\ifield_{\cbM\setminus\obM})$. For operations, we use the same notations as in the case of good spaces.

Recall (see \cite[Proposition 3.3.19]{DK16}\footnote
{The statement of this proposition is erroneous. 
The first isomorphism in loc.\ cit.\ may not hold
under the condition that $\oo f$  is topologically submersive.
However, it holds if $f$ is borderly submersive.
The second isomorphism, i.e.\ \eqref{eq:tsepb}, holds under the condition that $\oo f$  is topologically submersive.}) that
\eq
&&\text{$\reeim f\simeq\roim f$\quad if $f$ is proper,}\\
&&\text{$\epb f\simeq\epb f \field_{\obN}\tens\opb f$\quad
if $f\colon\bbM\to\bbN$ is  borderly  submersive.}
\label{eq:tsepb}
\eneq
The last statement implies
\eq
&&\text{$\epb f$ commutes with $\alpha$
if $f$ is  borderly  submersive.}\label{st:tsepb} 
\eneq
 With notations \eqref{eq:k}, \eqref{eq:tsepb} implies that
\begin{equation}\label{eq:i*!}
\opb i_\bbM \simeq \epb{i_\bbM},\quad  \opb j_\bbM \simeq \epb{j_\bbM}. 
\end{equation}

The quotient functor $\BDC(\ifield_\cbM)\to \BDC(\ifield_\bbM)$ 
is isomorphic to $\opb j_\bbM\simeq\epb{j_\bbM}$
and has a left  adjoint $\reeim{{j_\bbM}}$ and a right adjoint $\roim{{j_\bbM}}$, both fully faithful.

There is a natural embedding $$\iota_\bbM\colon\BDC(\field_\obM)\simeq\BDC(\field_\bbM)\to\BDC(\ifield_\bbM)$$ induced by $\iota_\cbM$. It has a left adjoint 
$$\alpha_\bbM\cl \BDC(\ifield_\bbM)\to \BDC(\field_\obM),$$
which in turn has a left adjoint $\beta_\bbM$.
One sets 
$\rhom\defeq\alpha_\bbM\rihom$, a functor with values in $\BDC(\field_\obM)$.

For $F\in \BDC(\field_\obM)$,
we often simply write  $F$ instead of $\iota_\bbM F$
in order to make notations less heavy. 

The functors $\iota_\bbM$, $\alpha_\bbM$ and $\beta_\bbM$ are exact. Moreover, $\iota_\bbM$ and $\beta_\bbM$ are fully faithful. This was shown in \cite{KS01} in the case of good spaces. The general case reduces to the former by the

\begin{lemma}\label{lem:iab}
One has
\begin{itemize}
\item[(i)]
 $\iota_\bbM \simeq \opb j_\bbM \,\iota_\cbM \,\roim{{\inc}} \simeq \roim{{i_\bbM}}\, \iota_\obM$,
\item[(ii)] 
$\alpha_\bbM \simeq \opb\inc \,\alpha_\cbM \,\reeim{{j_\bbM}} \simeq \alpha_\obM\; \opb i_\bbM$,
\item[(iii)] 
$\beta_\bbM \simeq \reeim{{i_\bbM}}\,  \beta_\obM$.
\end{itemize}
\end{lemma}

\begin{proof}
One has
\[
\opb j_\bbM\, \iota_\cbM \,\roim{{\inc}}
\underset{(*)}\simeq \opb j_\bbM \,\roim{{\inc}}\; \iota_\obM 
\simeq \opb j_\bbM \,\roim{{j_\bbM}} \,\roim{{i_\bbM}}\; \iota_\obM 
\simeq \roim{{i_\bbM}}\; \iota_\obM ,
\]
where $(*)$ follows from \eqref{eq:table}.

This proves (i). Then (ii) and (iii) follow by adjunction.
\end{proof}

For bordered spaces, the commutativity of the functor $\alpha$  with the operations is as follows.

\begin{lemma}\label{lem:alphaop}
Let $f\colon\bbM\to\bbN$ be a morphism of bordered spaces.
\begin{itemize}
\item[(i)] 
There are a natural isomorphism and a natural morphism of functors
\[
\oo f{}^{-1} \, \alpha_{\bbN} \simeq \alpha_{\bbM} \, \opb f , \qquad
 \alpha_{\bbM} \, \epb f \to \oo f{}^! \, \alpha_{\bbN} ,
\]
and the above morphism is an isomorphism if $f$ is  borderly  submersive.
\item[(ii)]
There are natural morphisms of functors
\[
 \reim {\oo f} \, \alpha_{\bbM} \to \alpha_{\bbN} \, \reeim f,  \qquad
\alpha_{\bbN} \, \roim f \to \roim {\oo f} \, \alpha_{\bbM},
\]
which are isomorphisms if $f$ is self-cartesian.
\item[(iii)]
For $K\in\BDC(\ifield_\bbM)$ and $L\in\BDC(\ifield_\bbN)$ one has
\[
\alpha_{\bbM\times\bbN}(K\etens L) \simeq (\alpha_\bbM K)\etens(\alpha_\bbN L).
\]
\end{itemize}
\end{lemma}

\begin{proof}[Proof of Lemma~\ref{lem:alphaop}]
(i-a)
By Lemma~\ref{lem:iab}~(ii) and \eqref{eq:table}, one has
$\oo f{}^{-1} \, \alpha_{\bbN} \simeq \oo f{}^{-1} \, \alpha_{\obN}\,\opb i_{\bbN} \simeq  \alpha_{\obM} \,\oo f{}^{-1} \,\opb i_{\bbN} \simeq  \alpha_{\obM}\, \opb i_{\bbM} \,\opb f \simeq \alpha_{\bbM} \, \opb f$'

\smallskip\noindent(i-b)
By Lemma~\ref{lem:iab}~(ii), the morphism is given by the composition
\[
\alpha_\obM \,\opb i_\bbM \,\epb f \isoto[(*)]
\alpha_\obM \, {\oo f}{}^{\;!} \, \opb i_\bbN \To[(**)]  {\oo f}{}^{\;!}\, \alpha_\obN \,\opb i_\bbN .
\]
Here, $(*)$ follows from \eqref{eq:i*!}, and
$(**)$ follows by adjunction from $\oo f{}^! \to \oo f{}^! \,\iota_\obN \,\alpha_\obN \simeq \iota_\obM \,\oo f{}^! \,\alpha_\obN$, with the isomorphism due to \eqref{eq:table}.

If $f$ is  borderly  submersive, $(**)$ is an isomorphism by \eqref{st:tsepb}.

\smallskip\noindent(ii-a)
 By Lemma~\ref{lem:iab}~(ii), the morphism is given by
\[
\reim{\oo f}\,\alpha_\obM \,\opb i_\bbM \simeq \alpha_\obN \,\reeim{\oo f}\, \opb i_\bbM \To[(*)] \alpha_\obN \,\opb i_\bbN \,\reeim f.
\]
Here $(*)$ follows by adjunction from $\reeim{{i_\bbN}} \, \reeim{\oo f}\,  \epb{i_\bbM} \simeq \reeim f\, \reeim{{i_\bbM}} \epb{i_\bbM} \to \reeim f$, recalling \eqref{eq:i*!}.

If $f$ is self-cartesian, this is an isomorphism by cartesianity.

\smallskip\noindent(ii-b) By Lemma~\ref{lem:iab}~(ii) and \eqref{eq:table}, the morphism is given by the composition
\[
\alpha_\obN \,\opb i_\bbN \,\roim f \to[(*)] \alpha_\obN\, \roim{\oo f}\, \opb i_\bbM \simeq \roim{\oo f}\,\alpha_\obM \,\opb i_\bbM.
\]
Here $(*)$ follows from  Lemma~\ref{lem:carsq}.

Recall \eqref{eq:i*!}. If $f$ is self-cartesian, then $(*)$ is an isomorphism by cartesianity.

\smallskip\noindent(iii) follows from $\alpha_\bbM \simeq \alpha_\obM \, \opb i_\bbM$ and \eqref{eq:table}.
\end{proof}

\subsection{Enhanced ind-sheaves}
Denote by $t\in\R$ the coordinate on the affine line, consider the two-point compactification $\overline\R\defeq \R\cup\st{-\infty,+\infty}$, and set $\bR\defeq(\R,\overline\R)$.
For $\bbM$ a bordered space, consider the projection 
\[
\pi_\bbM\colon \bbM\times\bR\to \bbM.
\]
Denote by $\BEC \bbM\defeq\BDC(\ifield_{\bbM\times\bR})/\opb\pi_\bbM\BDC(\ifield_\bbM)$ the bounded derived category of enhanced ind-sheaves of $\field$-vector spaces on $\bbM$. Denote by $\quot\colon\BDC(\ifield_{\bbM\times\bR})\to\BEC\bbM$ the quotient functor,
and by $\LE$ and $\RE$ its left and right adjoint, respectively.
They are both fully faithful.

For $f\colon \bbM\to \bbN$ a morphism of bordered spaces, set
\[
f_\R \defeq f\times\id_{\bR}\colon \bbM\times\bR\to \bbN\times\bR.
\]
Denote by $\ctens$, $\Eopb f$, $\Eeeim f$ and $\cihom$, $\Eoim f$, $\Eepb f$ the six operations for enhanced ind-sheaves.  
Recall that $\ctens$ is the additive convolution in the $t$ variable, and that the external operations are induced via $\quot$ by the corresponding operations for ind-sheaves, with respect to the morphism $f_\R$.
Denote by $\cetens$ the exterior tensor and by $\Edual$ the Verdier dual.

We have
\eq
&&\LE \, \quot (F)\simeq (\field_{\st{t\ge0}}\oplus
\field_{\st{t\le0}})\ctens F\qtq\\
&&\RE \, \quot (F)\simeq \cihom(\field_{\st{t\ge0}}\oplus
\field_{\st{t\le0}},F).
\eneq

The functors $\fihom$ and $\fhom$, taking values in $\BDC(\ifield_\bbM)$ and $\BDC(\field_\obM)$, respectively, are defined by
\begin{align} 
\label{eq:fihom}
\fihom(K_1, K_2) & \defeq \roim{{\pi_\bbM}}\rihom(F_1,\RE K_2) \\ \notag
& \ \simeq \roim{{\pi_\bbM}}\rihom(\LE K_1,F_2), \\
\label{eq:fhom}
\fhom(K_1, K_2) & \defeq \alpha_\bbM\fihom(K_1, K_2),
\end{align}
for $K_i\in\BEC\bbM$ and $F_i\in\BDC(\ifield_{\bbM\times\bR})$ such that $K_i=\quot F_i$ ($i=1,2$).

There is a natural decomposition $\BEC \bbM \simeq \BECp \bbM \dsum \BECm \bbM$,
given by $K\mapsto(\quot\,\field_{\st{t\geq 0}}\ctens K)\dsum(\quot\,\field_{\st{t\leq 0}}\ctens K)$.

There are embeddings
\[
\epsilon^\pm_\bbM\colon\BDC(\ifield_\bbM) \rightarrowtail \BECpm\bbM, \quad
F\mapsto \quot(\field_{\st{\pm t\geq 0}}\tens\opb{\pi_\bbM}F),
\]
and one sets $\epsilon_\bbM(F)\defeq\epsilon_\bbM^+(F)\dsum\epsilon_\bbM^-(F)\in\BEC\bbM$.
Note that $\epsilon_\bbM(F)\simeq\quot(\field_{\st{t= 0}}\tens\opb{\pi_\bbM}F)$.

\subsection{Stable objects}\label{sse:stable}
Let $\bbM$ be a bordered space.
Set
\begin{align*}
\field_{\{t\gg0\}} &\defeq \indlim[a\rightarrow+\infty]\field_{\{t\geq a\}} \in\BDC(\ifield_{\bbM\times\bR}), \\
\Efield_\bbM &\defeq \quot\field_{\{t\gg0\}} \in \BECp\bbM.
\end{align*}
An object $K\in\BECp\bbM$ is called \emph{stable} if $\Efield_\bbM\ctens K \isoto  K$.
We denote by $\BECs\bbM$ the full subcategory of $\BECp\bbM$ of stable objects.
The embedding $\BECs\bbM \rightarrowtail \BECp\bbM$ has a left adjoint $\Efield_\bbM\ctens\ast$, as well as a right adjoint $\cihom(\Efield_\bbM,\ast)$.

There is an embedding
\[
e_\bbM\colon \BDC(\ifield_\bbM) \rightarrowtail \BECs\bbM, \quad F\mapsto \Efield_\bbM \ctens \epsilon_\bbM(F)
\simeq \quot(\field_{\{t\gg0\}}\tens\opb{\pi_\bbM} F).
\]

\begin{notation}\label{no:ex}
Let $S\subset T$ be locally closed subsets of $\bbM$.
\begin{itemize}
\item[(i)]
For continuous  maps $\varphi_\pm\colon  T\to\overline\R$ such that
$-\infty\leq \varphi_-\le\varphi_+<+\infty$, set
\begin{align*}
\ex_{S|\bbM}^{\range{\varphi_+}{\varphi_-}} &\seteq \quot\,\field_{\st{x\in S,\  -\varphi_+(x) \leq t < -\varphi_-(x)}} \quad \in \BECp \bbM, \\ \notag
\Ex_{S|\bbM}^{\range{\varphi_+}{\varphi_-}} &\defeq \Efield_\bbM \ctens \ex_{S|\obM}^{\range{\varphi_+}{\varphi_-}}  \quad \in  \BECs \bbM ,
\end{align*}
where we write for short
\[ 
\begin{split}
\{x\in S,&\ -\varphi_+(x) \leq t < -\varphi_-(x)\} \\ 
&\defeq \st{(x,t)\in \obM\times\R\semicolon  x\in S,\ -\varphi_+(x) \leq t < -\varphi_-(x)},
\end{split}
\]
with $<$ the total order on $\cR$.  If $S=T$, we also write for short
\[
\{-\varphi_+(x) \leq t < -\varphi_-(x)\} \defeq
\{x\in T,\ -\varphi_+(x) \leq t < -\varphi_-(x)\}.
\]
\item[(ii)]
For a continuous map $\varphi\colon  T\to\R$,
consider the object of $\BECp \bbM$
\begin{align*}
\ex_{S|\bbM }^\varphi &\seteq \quot\field_{\st{ x\in S,\  t+\varphi(x)\geq 0}}
\quad \in \BECp \bbM, \\ \notag
\Ex_{S|\bbM}^\varphi &\defeq \Efield_\bbM \ctens \ex_{S|\obM}^\varphi
\quad \in \BECs\bbM.
\end{align*}
where we write for short 
\[
\st{ x\in S,\  t+\varphi(x)\geq 0} = \st{(x,t)\in \obM\times\R\semicolon  x\in S,\  t+\varphi(x)\geq 0}.
\]
 If $S=T$, we also write for short
\[
\st{t+\varphi(x)\geq 0} \defeq
\st{x\in T,\ t+\varphi(x)\geq 0}.
\]
\end{itemize}
\end{notation}

Note that one has
$\ex_{S| \bbM}^\varphi\simeq
\ex_{S|\bbM}^{\range{\varphi}{-\infty}} $, and that there is a short exact sequence
\[
0\to \ex_{S|\bbM}^{\range{\varphi_+}{\varphi_-}} \to \ex_{S| \bbM}^{\varphi_+} \to \ex_{S|\bbM}^{\varphi_-} \to 0
\]
in the heart of $\BEC \bbM$ for the natural $t$-structure.

\subsection{Constructible objects}
A subanalytic bordered space is a bordered space $\bbM$ such that $\obM$ is an open subanalytic subset of the subanalytic space $\cbM$. A morphism $f\colon\bbM\to\bbN$ of subanalytic bordered spaces is a morphism of bordered spaces such that $\Gamma_{\oo f}$ is subanalytic in $\cbM\times\cbN$.
By definition, a subset $Z$ of $\bbM$ is subanalytic if it is subanalytic in $\cbM$. 

\smallskip
Let $\bbM$ be a subanalytic bordered space.

Denote by $\BDC_\wRc(\field_\bbM)$ the full subcategory of $\BDC(\field_\obM)$ whose objects $F$ are such that $\roim{{\inc}} F$ (or equivalently, $\reim{{\inc}} F$) is weakly $\R$-constructible, for $\inc\colon \obM\to\cbM$ the embedding.  
We similarly define the category $\BDC_\Rc(\field_\bbM)$ of $\R$-constructible sheaves.

Denote by
$\BECwRc\bbM$ the strictly full subcategory of $\BEC\bbM$ whose objects $K$
are such that for any relatively compact open subanalytic subset $U$ of $\bbM$,
one has
\[
\opb{\pi_\bbM}\field_U\tens K \simeq e_\bbM F
\]
for some $F\in \BDC_\wRc(\field_{\bbM\times\bR})$.
In particular, $K$ belongs to $\BECs\bbM$. 
We similarly define the category $\BECRc\bbM$ of $\R$-constructible enhanced ind-sheaves.

\section{Sheafification}\label{se:sh}

In this section, we discuss what we call here ind-sheafification and sheafification functor, and prove some of their functorial properties. Concerning constructibility, we use a fundamental result from \cite[\S6]{KS16D}.

\subsection{Associated ind-sheaf}

Let $\bbM$ be a bordered space.
Let $i_0\colon\bbM\to\bbM\times\bR$ be the embedding $x\mapsto(x,0)$.

\begin{definition}
Let $K\in\BEC \bbM$ and take
 $F\in\BDC(\ifield_{\bbM\times\bR})$ such that $K\simeq\quot F$.
We set
\eqn
\Ish_\bbM(K) &&\defeq \fihom(\quot\,\field_{\st{t=0}},K) \\
&&\simeq \roim{{\pi_\bbM}}\rihom(\field_{\st{t\ge0}}\oplus\field_{\st{t\le0}}, F) \\
&&\simeq\roim{{\pi_\bbM}}\rihom(\field_{\st{t=0}}, \RE K) \\
&&\simeq \reeim{{\pi_\bbM}}\rihom(\field_{\st{t=0}}, \RE K) \\
&&\simeq \epb{i_0}\RE K \quad \in \BDC(\ifield_\bbM)
\eneqn
(see \cite[Lemma 4.5.16]{DK16}), and call it the \emph{associated ind-sheaf} (in the derived sense) to $K$ on $\bbM$.
We will write for short $\Ish=\Ish_\bbM$, if there is no fear of confusion.
\end{definition}

Note that one has
\begin{align*}
\Ish(K) &\simeq \fihom(\quot\,\field_{\st{t\geq0}},K)
&\text{for $K\in\BECp \bbM$,} \\
\Ish(K) &\simeq \fihom(\field_\bbM^\enh,K)
&\text{for $K\in\BECs \bbM$.}
\end{align*}

\begin{lemma}\label{lem:ishadj}
The following are pairs of adjoint functors
\begin{itemize}
\item[(i)]
$\xymatrix{
(\epsilon,\Ish)\colon  \BDC(\ifield_\bbM) \ar@<.5ex>[r]^-\epsilon & \BEC \bbM  \ar@<.5ex>[l]^-\Ish},$
\item[(ii)]
$\xymatrix{( \epsilon^+,\Ish)\colon  \BDC(\ifield_\bbM) \ar@{ >->}@<.5ex>[r]^-{\epsilon^+} & \BECp \bbM  \ar@<.5ex>[l]^-\Ish},$
\item[(iii)]
$\xymatrix{(e,\Ish)\colon  \BDC(\ifield_\bbM) \ar@{ >->}@<.5ex>[r]^-e & \BECs \bbM  \ar@<.5ex>[l]^-\Ish}.$
\end{itemize}
\end{lemma}

\begin{proof}
(i)
For $F\in\BDC(\ifield_\bbM)$ and $K\in\BEC \bbM$ one has
\begin{align*}
\Hom[\BEC \bbM](\epsilon(F),K)
&\simeq \Hom[\BDC(\ifield_{\bbM\times\bR})](\opb\pi F\tens\field_{\st{t=0}},\RE K) \\
&\simeq \Hom[\BDC(\ifield_\bbM)](F,\roim\pi\rihom(\field_{\st{t=0}},\RE K)) \\
&\simeq \Hom[\BDC(\ifield_\bbM)](F,\Ish(K)).
\end{align*}

\smallskip\noindent(ii) and (iii) follow from (i), noticing that there are pairs of adjoint functors $(\ast\ctens\quot\field_{\st{t\geq0}},\iota)$ and $(\ast\ctens\field_\bbM^\enh,\iota)$:
\[
\xymatrix@C=5em{
\BEC \bbM \ar@<.5ex>[r]^-{\ast\ctens\quot\field_{\st{t\geq0}}} & \BECp \bbM  \ar@<.5ex>@{ >->}[l]^\iota  \ar@<.5ex>[r]^-{\ast\ctens\field_\bbM^\enh} & \BECs \bbM  \ar@<.5ex>@{ >->}[l]^\iota.
}
\] 
Here we denote by $\iota$ the natural embeddings.
\end{proof}

\begin{lemma}\label{lem:ishop}
Let $f\colon \bbM\to \bbN$ be a morphism of bordered spaces.
\begin{itemize}
\item[(i)] There are a natural morphism and a natural isomorphism of functors
\[
\opb f \, \Ish_{\bbN} \to \Ish_{\bbM} \, \Eopb f, \quad
\epb f \, \Ish_{\bbN} \simeq \Ish_{\bbM} \, \Eepb f,
\]
and the above morphism is an isomorphism if $f$ is borderly submersive.
\item[(ii)] There are a natural morphism and a natural isomorphism of functors
\[
\reeim f \, \Ish_{\bbM} \to \Ish_{\bbN} \, \Eeeim f, \quad
\roim f \, \Ish_{\bbM} \simeq \Ish_{\bbN} \, \Eoim f,
\]
and the above morphism is an isomorphism if $f$ is proper.
\item[(iii)] For $K\in\BEC\bbM$ and $L\in\BEC\bbN$,
there is a natural morphism
\[
\Ish(K)\etens\Ish(L)\to\Ish(K\cetens L).
\]
\end{itemize}
\end{lemma}

\begin{proof}
Recall that one sets $f_\R \defeq f\times\id_\bR\colon \bbM\times\bR\to \bbN\times\bR$.

\smallskip\noindent(i) Let $L\in\BEC\bbN$ and set $G\defeq\RE L\in\BDC(\ifield_{\bbN\times\bR})$.

\smallskip\noindent(i-a) One has
\begin{align*}
\opb f \Ish_{\bbN}(L)
&\simeq \opb f \,\reeim{{\pi_\bbN}} \rihom(\field_{\st{t=0}},G) \\
&\simeq\reeim{{\pi_\bbM}} \,\opb f_\R \rihom(\field_{\st{t=0}},G) \\
&\underset{(*)}\to \roim{{\pi_\bbM}} \rihom(\field_{\st{t=0}},\opb f_\R G) \\
&\underset{(**)}\to \roim{{\pi_\bbM}} \rihom(\field_{\st{t=0}},\RE\Eopb f L) \\
&\simeq \Ish_{\bbM}(\Eopb f L).
\end{align*}
Here, $(*)$ follows from \cite[Proposition 3.3.13]{DK16},
and $(**)$ from Lemma~\ref{lem:A3}.

If $f$ is borderly submersive,
then $(*)$ is an isomorphism by \cite[Proposition 3.3.19]{DK16}
and $(**)$ is an isomorphism by Lemma~\ref{lem:A3}.

\smallskip\noindent(i-b) Recall that $\epb{f_\R}G \simeq \RE(\Eepb f L)$. One has
\begin{align*}
\epb f \Ish_{\bbN}(L)
&= \epb f \roim{{\pi_\bbN}} \rihom(\field_{\st{t=0}},G) \\
&\simeq \roim{{\pi_\bbM}}\, \epb{f_\R} \rihom(\field_{\st{t=0}},G) \\
&\simeq \roim{{\pi_\bbM}} \rihom(\field_{\st{t=0}},\epb{f_\R} G) \\
&\simeq \roim{{\pi_\bbM}} \rihom(\field_{\st{t=0}},\RE(\Eepb f L)) \\
&\simeq \Ish_{\bbM} (\Eepb f L).
\end{align*}

\smallskip\noindent(ii)
Let $K\in\BEC\bbM$ and set $F\defeq\RE K\in\BDC(\ifield_{\bbM\times\bR})$.

\smallskip\noindent(ii-a)
One has
\begin{align*}
\Ish_{\bbN}(\Eeeim f K)
&=\reeim{{\pi_\bbN}} \rihom(\field_{\st{t=0}}, \RE\,\Eeeim f K) \\
&\from \reeim{{\pi_\bbN}} \rihom(\field_{\st{t=0}}, \reeim{{f_\R}}F) \\
&\underset{(*)}\simeq \reeim{{\pi_\bbN}}\, \reeim{{f_\R}}\rihom(\field_{\st{t=0}}, F) \\
&\isofrom \reeim f\,\reeim{{\pi_\bbM}} \rihom(\field_{\st{t=0}}, F) \\
&= \reeim f (\Ish_{\bbM}(K)).
\end{align*}
Here $(*)$ follows from \cite[Lemma 5.2.8]{KS01}.

\smallskip\noindent(ii-b)
Since $\RE(\Eoim f K) \simeq \roim{{f_\R}}F$, one has
\begin{align*}
\Ish_{\bbN}(\Eoim f K)
&\simeq \roim{{\pi_\bbN}} \rihom(\field_{\st{t=0}}, \roim{{f_\R}}F) \\
&\simeq \roim{{\pi_\bbM}}\, \roim{{f_\R}}\rihom(\field_{\st{t=0}}, F) \\
&\simeq \roim f\,\roim{{\pi_\bbM}} \rihom(\field_{\st{t=0}}, F).
\end{align*}
If $f$ is proper, $\eim f\simeq\oim f$.

\smallskip\noindent(iii)
Set $F\defeq\RE K\in\BDC(\ifield_{\bbM\times\bR})$ and $G\defeq\RE L\in\BDC(\ifield_{\bbN\times\bR})$.
Recall that $F\cetens G \defeq \reeim m(F\etens G)$, where
\[
m\colon \bbM\times\bR \times\bbN\times\bR \to \bbM\times\bbN\times\bR \quad
(x,t_1,y,t_2)\mapsto (x,y,t_1+t_2).
\]
Then, one has
\begin{align*}
\Ish(K)&{}\etens\Ish(L)\\
&\simeq \roim{{\pi_\bbM}}\rihom(\field_{\st{t_1=0}}, F) \etens \roim{{\pi_\bbN}}\rihom(\field_{\st{t_2=0}}, G) \\
&\to \roim{(\pi_\bbM\times\pi_\bbN)}\bl\rihom(\field_{\st{t_1=0}}, F) \etens \rihom(\field_{\st{t_2=0}}, G)\br \\
&\to \roim{{\pi_{\bbM\times\bbN}}}\roim m\rihom(\field_{\st{t_1=0}}\etens\field_{\st{t_2=0}}, F\etens G) \\
&\to \roim{{\pi_{\bbM\times\bbN}}}\rihom\bl
\reeim m(\field_{\st{t_1=0}}\etens\field_{\st{t_2=0}}), \reeim m (F\etens G)\br \\
&\simeq \roim{{\pi_{\bbM\times\bbN}}}\rihom(\field_{\st{t=0}}, F\cetens G),
\end{align*}
One concludes using the natural morphism $F\cetens G \to \RE(K\cetens L)$.
\end{proof}

\medskip
\subsection{Associated sheaf}

Let $\bbM$ be a bordered space.

\begin{definition}
Let $K\in\BEC \bbM$. 
\begin{itemize}
\item[(i)]
We set
\begin{align*}
\sh_\bbM(K) &\defeq \fhom(\quot\,\field_{\st{t=0}},K) \\
&= \alpha_\bbM\,\Ish_\bbM(K) \quad \in \BDC(\field_\obM),
\end{align*}
and call it the \emph{associated sheaf} (in the derived sense) to $K$ on $\obM$.
We will write for short $\sh=\sh_\bbM$, if there is no fear of confusion.
\item[(ii)]
We say that $K$ is of \emph{sheaf type} (in the derived sense) if it is in the essential image of
\[
e_\bbM\,\iota_\bbM\colon \BDC(\field_\obM) \rightarrowtail\BEC\bbM,
\]
\end{itemize}
\end{definition}

One has
\begin{align*}
\sh_\bbM(K) &\simeq \fhom(\quot\,\field_{\st{t\geq0}},K),
&\text{for } K\in\BECp \bbM, \\
\sh_\bbM(K) &\simeq \fhom(\field^\enh_\bbM,K),
&\text{for } K\in\BECs \bbM.
\end{align*}

\begin{lemma}\label{lem:shi}
One has $\sh_\bbM\simeq \sh_\obM\, \Eopb i_\bbM$.
\end{lemma}

\begin{proof}
Recall that $\opb {i_\bbM}\simeq\epb {i_\bbM}$. Using Lemma~\ref{lem:iab}~(ii), one has
\begin{align*}
\alpha_\bbM\,\Ish_\bbM
&\simeq \alpha_\obM\,\epb{i_\bbM} \fihom(\quot\,\field_{\st{t=0}},K) \\
&= \alpha_\obM\,\epb{i_\bbM}\,\roim{{\pi_\bbM}}\rihom(\field_{\st{t=0}},\RE K) \\
&\simeq \alpha_\obM\,\roim{{\pi_\obM}}\,\epb{i_{\bbM\times\bR}} \rihom(\field_{\st{t=0}},\RE K) \\
&\simeq \alpha_\obM\,\roim{{\pi_\obM}}\rihom(\field_{\st{t=0}},\epb{i_{\bbM\times\bR}}\RE K) \\
&\simeq \alpha_\obM\,\roim{{\pi_\obM}}\rihom(\field_{\st{t=0}},\RE \,\Eepb i_\bbM K) \\
&\simeq \alpha_\obM\,\fihom(\quot\field_{\st{t=0}},\Eepb i_\bbM K) .
\end{align*}
\end{proof}

Let $\bbM$ be a bordered space, and consider the natural morphisms
of good spaces
\[
\obM\times\R\to[k] \obM\times\cR\to[\overline\pi] \obM.
\]
 We write $\ct$ for points of $\cR\defeq \R\cup\st{-\infty,+\infty}$.

An important tool in this framework is given by 

\begin{proposition}[{\cite[Corollary 6.6.6]{KS16D}}] \label{pro:KScor}
Let $\bbM$ be a bordered space.
Then, for $F \in\BDC(\field_{\obM\times\R})$ one has 
\[
\sh_\bbM(\field_\bbM^\enh \ctens \quot F) 
\simeq \roim{\overline\pi}(\field_{\st{-\infty<\ct\leq+\infty}}\tens\roim k F).
\]
\end{proposition}

Denote by $i_{\pm\infty}\colon M\to M\times\cR$ the embeddings $x\mapsto(x,\pm\infty)$.
Using the above proposition and \cite[Proposition 4.3.10, Lemma 4.3.13]{DK16}, we get

\begin{corollary}\label{cor:shother}
Let $\bbM$ be a bordered space.
Then, for $F \in\BDC(\field_{\obM\times\R})$ one has
\begin{align*}
\sh_\bbM(\field_\bbM^\enh \ctens \quot F) 
&\simeq \opb i_{+\infty} \roim j\LEp \quot F\\ 
&\simeq \opb i_{-\infty} \roim j\REp \quot F[-1] \\
&\simeq \roim\pi\LEp \quot F\\ 
&\simeq \reim\pi\REp \quot F.
\end{align*}
\end{corollary}

Consider the functors 
\begin{equation}
\label{eq:she}
\xymatrix@C=4em{
\BDC(\field_\obM) \ar@{ >->}@<.5ex>[r]^-{e_\bbM\,\iota_\bbM} & \BEC \bbM  \ar@<.5ex>[l]^-{\sh_\bbM}.
}
\end{equation}
As explained in the Introduction, $(e_\bbM\,\iota_\bbM,\;\sh_\bbM)$ is \emph{not} an adjoint pair of functors in general.

\begin{proposition}\label{pro:sheid}
Consider the functors \eqref{eq:she}.
\begin{itemize}
\item[(i)]
$\sh_\bbM$ is a left quasi-inverse to $e_\bbM\,\iota_\bbM$.
\item[(ii)]
The property of being of sheaf type is local\,\footnote{ Saying that a property $\shp(\bbM)$ is local on $\bbM$ means the following. For any open covering $\st{ U_i}_{i\in I}$ of $\bbM$, $\shp(\bbM)$ is true if and only if $\shp\bl\inb{(U_i)}\br$ is true for any $i\in I$.}
on $\bbM$,
and $K\in\BEC \bbM$ is of sheaf type if and only if 
$K\simeq e_\bbM\,\iota_\bbM\bl\sh_\bbM(K)\br$.
\end{itemize}
\end{proposition}

\begin{proof}
(i) 
By Proposition~\ref{pro:KScor}, for $L\in\BDC(\field_\obM)$, one has
\begin{align*}
\sh_\bbM\, e_\bbM\,\iota_\bbM (L)
&\simeq \sh_\bbM\bl\field_\bbM^\enh \ctens \quot (\field_{\st{t=0}}\tens \pi^{-1}\iota_\bbM L)\br\\
&\simeq \roim{\overline\pi}
\bl\field_{\st{-\infty< \ct\leq+\infty}}\tens\field_{\st{\ct=0}}\tens
\opb{\overline\pi} L\br \\
&\simeq \reim{\overline\pi}
\bl\field_{\st{\ct=0}}\tens
\opb{\overline\pi} L\br \\
&\simeq \bl\reim{\overline\pi}\field_{\st{\ct=0}}\br\tens L \simeq L.
\end{align*}

\smallskip\noindent(ii) follows from (i).
\end{proof}

By Lemmas~\ref{lem:alphaop} and \ref{lem:ishop}, one gets

\begin{lemma}\label{lem:shop}
Let $f\colon \bbM\to \bbN$ be a morphism of bordered spaces.
\begin{itemize}
\item[(i)] There are natural morphisms of functors
\[
\opb {\smash{\oo f}} \, \sh_{\bbN} \to \sh_{\bbM} \, \Eopb f, \quad
 \sh_{\bbM} \, \Eepb f \to \epb {\smash{\oo f}} \, \sh_{\bbN}. ,
\]
which are isomorphisms if $f$ is  borderly submersive.
\item[(ii)] There are natural morphisms of functors
\[
\reim{\oo f} \, \sh_{\bbM} \to \sh_{\bbN} \, \Eeeim f, \quad
\sh_{\bbN} \, \Eoim f \to \roim{\oo f} \, \sh_{\bbM} .
\]
 The first morphism is an isomorphism  if $f$ is proper. 
The second morphism is an isomorphism if $f$ is self-cartesian,
 and in particular if $f$ is proper.
\item[(iii)] For $K\in\BEC\bbM$ and $L\in\BEC\bbN$,
there is a natural morphism
\[
\sh(K)\etens\sh(L)\to\sh(K\cetens L).
\]
\end{itemize}
\end{lemma}

\begin{example}
Let $M=\R_x$, $U=\st{x>0}$.
By Corollary~\ref{cor:shother} one has
\begin{align*}
\RE\ex^{1/x}_{U|M} &\simeq \field_{\st{x>0,\ xt<-1}}[1], &
\RE\ex^{-1/x}_{U|M} &\simeq \field_{\st{x\geq0,\ xt<1}}[1], \\
\sh(\Ex^{1/x}_{U|M}) &\simeq \field_{\st{x>0}}, &
\sh(\Ex^{-1/x}_{U|M}) &\simeq \field_{\st{x\geq0}}.
\end{align*}
Note that, denoting by $i\colon\st0\to M$ the embedding, one has
\[
\epb i (\sh (\Ex^{1/x}_{U|M})) \not\simeq \sh (\Eepb i (\Ex^{1/x}_{U|M})), \quad
\opb i (\sh (\Ex^{-1/x}_{U|M})) \not\simeq \sh (\Eopb i (\Ex^{-1/x}_{U|M})).
\]
In fact, on one hand one has $\epb i (\sh (\Ex^{1/x}_{U|M}))\simeq\field[-1]$ and $\Eepb i (\Ex^{1/x}_{U|M})\simeq 0$, and on the other hand one has $\opb i (\sh (\Ex^{-1/x}_{U|M}))\simeq\field$ and $\Eopb i (\Ex^{-1/x}_{U|M})\simeq0$.

Note also that $\sh$ is \emph{not} conservative, since $\sh(\Ex_{U|X}^{\range{2/x}{1/x}})\simeq 0$.
\end{example}

\begin{example}
Let $X\subset \C_z$ be an open neighborhood of the origin, and set $\bdot X=X\setminus\st0$. 
The real oriented blow-up $p\colon \RB[0]X \to X$ with center the origin is
defined by $\RB[0]X \defeq\st{(r,w)\in\R_{\ge0}\times\C\semicolon |w|=1,\ rw\in X}$,
$p(r,w)=r w$. Denote by $S_0X=\st{r=0}$ the exceptional divisor.

Let $f\in\O_X(*0)$ be a meromorphic function with pole order $d>0$
at the origin. 
With the identification $\bdot X\simeq\st{r>0}\subset\RB[0]X$, the set 
$I\seteq S_0X\setminus\ol{\st{z\in\bdot X\;;\;\Re f(z)\ge0}}$
is the disjoint union of $d$ open non-empty intervals.
Here 
$\ol{\st{\cdot}}$ is the closure in $\RB[0]X$.
Then, recalling Notation~\ref{no:ex},
\[
\sh(\Ex_{\bdot X|X}^{\Re f}) \simeq 
\sh (\Eoim p \Ex_{\bdot X|\RB[0]X}^{\Re f\circ p})
\simeq \roim p\sh ( \Ex_{\bdot X|\RB[0]X}^{\Re f\circ p})\simeq
\reim p \field_{I\sqcup\bdot X}.
\]
 Recall  that, for $\field=\C$, the Riemann-Hilbert correspondence of \cite{DK16} associates the meromorphic connection $d-df$ with $\Ex_{\bdot X|X}^{\Re f}$ by the functor $\drE$. 
\end{example}

\subsection{(Weak-)\,constructibility}

An important consequence of Proposition~\ref{pro:KScor} is

\begin{proposition}[{\cite[Theorem 6.6.4]{KS16D}}]
Let $\bbM$ be a subanalytic bordered space.
The functor $\sh_\bbM$ induces functors
\begin{align*}
\sh_\bbM&\colon\BECwRc \bbM \to \BDC_\wRc(\field_\bbM), \\
\sh_\bbM&\colon\BECRc \bbM \to \BDC_\Rc(\field_\bbM). 
\end{align*}
\end{proposition}

\begin{proposition}
Let $\bbM$ be a subanalytic bordered space.
For $K\in\BECRc \bbM$ there is a natural isomorphism
\[
\sh_\bbM( \Edual_\bbM K) \isoto \dual_\obM( \sh_\bbM K).
\]
\end{proposition}

\begin{proof}
Recall that $\sh_\bbM \simeq \sh_\obM \, \Eopb{i_\bbM}$
and $\Eopb{i_\bbM}\simeq\Eepb{i_\bbM}$.
Since $\Eopb{i_\bbM} \, \Edual_\bbM  \simeq \Edual_\obM \, \Eopb{i_\bbM}$,
we may assume that $\bbM=\obM=M$ is a subanalytic space.

\smallskip\noindent(i)
Let us construct a natural morphism
\[
\sh( \Edual K) \to \dual( \sh K).
\]
By adjunction, it is enough to construct a natural morphism
\[
\sh( \Edual K)\tens\sh(K) \to \omega_M.
\]
Note that we have a morphism
\[
\Edual K\ctens K \to \omega_M^\enh.
\]
Let $\delta\colon M\to M\times M$ be the diagonal embedding, so that
$\Edual K\ctens K\simeq\Eopb\delta(\Edual K\cetens K)$.
There are natural morphisms
\begin{align*}
\sh(\Edual K)\tens\sh(K)
&\simeq \opb\delta\bl\sh(\Edual K)\etens\sh(K)\br \\
&\underset{(*)}\to \opb\delta\bl\sh(\Edual K\cetens K)\br \\
&\underset{(**)}\to \sh\bl\Eopb\delta(\Edual K\cetens K)\br \\
&\to \sh(\omega^\enh_M) \simeq \omega_M,
\end{align*}
where $(*)$ is due to  Lemma~\ref{lem:shop}~(iii), and $(**)$ is due to Lemma~\ref{lem:shop}~(i).

\smallskip\noindent(ii)
By (i), the problem is local on $M$. Hence, we may assume that
$K\simeq\field_M^\enh \ctens \quot F$ for $F\in\BDC_\Rc(\field_{M\times\bR})$.
Considering the morphisms
\[
 k\colon  M\times\bR \to[i^\pm] M\times(\R\union\st{\pm\infty},\overline\R) \to[j^\pm] M\times\overline\R.
\]
 Since
\[
\field_{\st{-\infty<\overline t\leq+\infty}}\tens\roim k F 
\simeq \reim{j^+}\,\roim{i^+} F
\simeq \roim{j^-}\,\reim{i^-} F,
\]
Proposition~\ref{pro:KScor} gives
\begin{align*}
\sh_M(K) 
&\simeq \roim{\overline\pi}\,\reim{j^+}\,\roim{i^+} F \\
&\simeq \roim{\overline\pi}\,\roim{j^-}\,\reim{i^-} F .
\end{align*}
 By \cite[Proposition~4.8.3]{DK16} one has  $\Edual_M(\field_M^\enh \ctens \quot F) \simeq \field_M^\enh \ctens \quot \opb a\dual_{M\times\bR} F$, where $a\colon M\times\bR\to M\times\bR$ is given by $a(x,t)=(x,-t)$.
Then, one has
\begin{align*}
\sh_M(\Edual_M K) 
&\simeq \sh_M(\field_M^\enh \ctens \quot \,\opb a\,\dual_{M\times\bR} F) \\
&\simeq \roim{\overline\pi}\,\reim{j^+}\,\roim{i^+} \,\opb a\,\dual_{M\times\bR} F \\
&\simeq \roim{\overline\pi}\,\reim{j^-}\,\roim{i^-} \,\dual_{M\times\bR} F \\
&\simeq \dual_M (\roim{\overline\pi}\,\roim{j^-}\,\reim{i^-} F) \\
&\simeq \dual_M (\sh_M(K)).
\end{align*}
\end{proof}

\begin{lemma}\label{lem:shetens}
Let $\bbM$ and $\bbN$ be bordered spaces.
Let $F\in\BDC_\Rc(\field_\bbM)$ and $L\in\BEC\bbN$. Then
\[
\sh(\epsilon(F)\cetens L) \simeq F \etens \sh(L).
\]
\end{lemma}

\begin{proof}
For $G\defeq\RE L\in\BDC(\ifield_{\bbN\times\bR})$, one has
\begin{align*}
\sh(\epsilon(F)\cetens L) 
&\simeq \alpha_{\bbM\times\bbN}\,\roim{{\pi_{\bbM\times\bbN}}}\,\rihom(\field_{\st{t\geq 0}}, F\etens G) \\
&\underset{(a)}\simeq \alpha_{\bbM\times\bbN}\,\roim{{\pi_{\bbM\times\bbN}}}
\bl F\etens\rihom(\field_{\st{t\geq 0}}, G)\br \\
&\underset{(b)}\simeq \alpha_{\bbM\times\bbN}\bl F\etens
\roim{{\pi_{\bbN}}}\rihom(\field_{\st{t\geq 0}}, G)\br \\
&\simeq  F\etens\alpha_\bbN\,\roim{{\pi_\bbN}}\rihom(\field_{\st{t\geq 0}}, G) ,
\end{align*}
where $(a)$ follows from \cite[Corollary 2.3.5]{DK16}
and $(b)$ follows from Proposition~\ref{prop:extd} in Appendix.
\end{proof}

\section{Germ formula}\label{se:germ}

As we saw in the previous section, sheafification does not commute with the pull-back by a closed embedding, in general. We provide here a germ formula for the sheafification of such a pull-back, using results from Appendix~\ref{se:wc}.

\subsection{Restriction and germ formula}

Let $\bbM$ be a subanalytic bordered space.
Recall Notation~\ref{no:ex}.

Let $N\subset\bbM$ be a closed subanalytic subset, denote by $i\colon \bN\to \bbM$ the embedding.
To illustrate the difference between $\sh\,\Eopb i$ and $\opb i \, \sh$ note that  on one hand, by \cite[Lemma~2.4.1]{DK18},  for $K\in\BECp \bbM$ and $y_0\in N$ one has\footnote{Recall from \cite[\S2.1]{DK18} that, for any $c,d\in\Z$, small filtrant inductive limits exist in $\derd^{[c,d]}(\field)$, the full subcategory of $\BDC(\field)$ whose objects $V$ satisfy $H^j(V)=0$ for $j<c$ or $j>d$. That is, uniformly bounded small filtrant inductive limits exist in $\BDC(\field)$.}
\begin{align*}
\bl \opb i \sh(K)\br_{y_0} &\simeq
\sh(K)_{y_0} \\ &\simeq
\ilim[U\owns y_0]\FHom(\ex^0_{U|\bbM},K),
\end{align*}
 where $U$ runs over the open neighborhoods of $y_0$ in $\obM$. 
On the other hand,

\begin{proposition}\label{pro:sgerm}
Let $\varphi\colon \bbM\to \bR$ be a morphism of subanalytic bordered spaces, set
$N \defeq \oo\varphi{}^{-1}(0)\subset\bbM$, and denote by $i\colon \bN\to \bbM$ the embedding.
For $y_0\in N$ and $K\in\BECwRc \bbM$ one has
\[
\sh(\Eopb i K)_{y_0} \simeq
\ilim[\substack{U\owns y_0\\\delta,\varepsilon\to0+}]\FHom(\exx^{\range0{-\delta|\oo\varphi(x)|^{-\varepsilon}}}_{U|\bbM},K),
\]
where $U$ runs over the open neighborhoods of $y_0$ in $\obM$.
Here, we set $-\delta|\oo\varphi(x)|^{-\varepsilon}=-\infty$ for $\oo\varphi(x)=0$.

More generally, for $T\subset N$ a compact subset one has
\[
\rsect(T;\sh(\Eopb i K)) \simeq
\ilim[\substack{U\supset T\\\delta,\varepsilon\to0+}]\FHom(\exx^{\range0{-\delta|\oo\varphi(x)|^{-\varepsilon}}}_{U|\bbM},K),
\]
where $U$ runs over the open neighborhoods of $T$ in $\obM$.
\end{proposition}

\begin{proof}
 Since $y_0\in N\subset \obM$, 
we may assume that $\bbM=\obM\eqdef M$ is a subanalytic space.

Since $\rsect(T;\sh(\Eopb i K)) \simeq
\ilim[\substack{U\supset T}]\rsect(U;\sh(\Eopb i K))$,
we may assume that $U$ runs over the open {\em subanalytic} neighborhoods of $T$ in $\obM$.

We will split the proof of the last isomorphism in the statement into three parts.

\smallskip\noindent(i)
Up to shrinking $M$ around $T$, we can assume that there exists $F\in\BDC_\wRc
(\field_{M\times\bR})$ such that $K \simeq \field_M^\enh\ctens \quot F$. 
For $c\in\R$, and $U$ an open relatively compact subanalytic subset of $M$ containing $T$, set
\[
U_{c,\delta,\varepsilon}  \defeq \st{(x,t)\in U\times\R \semicolon \,t+c<\delta|\varphi(x)|^{-\varepsilon}}.
\]
Note that $\LE \exx^{\range{c}{c-\delta|\varphi(x)|^{-\varepsilon}}}_{U|M} \simeq \field_{U_{c,\delta,\varepsilon}}\tens\field_{\st{ t\geq -c}}$.
Then, one has
\begin{align*}
\FHom(\exx^{\range0{-\delta|\varphi(x)|^{-\varepsilon}}}_{U|M},K)
&\simeq 
\ilim[c\to+\infty]\FHom(\quot\field_{\st{t\geq -c}}\ctens\exx^{\range0{-\delta|\varphi(x)|^{-\varepsilon}}}_{U|M},\quot F) \\
&\simeq 
\ilim[c\to+\infty]\Hom(\LE\exx^{\range{c}{c-\delta|\varphi(x)|^{-\varepsilon}}}_{U|M},F) \\
&\simeq 
\ilim[c\to+\infty]\Hom(\field_{U_{c,\delta,\varepsilon}}\tens\field_{\st{t\geq -c}}, F) \\
&\simeq \ilim[c\to+\infty]\Hom(\field_{U_{c,\delta,\varepsilon}},\field_{\st{t> -c}}\tens F) \\
&\simeq 
\ilim[c\to+\infty]\rsect\bl U_{c,\delta,\varepsilon};\;\field_{\st{t> -c}}\tens F)\\
&\simeq 
\ilim[c\to+\infty]\rsect\bl U_{c,\delta,\varepsilon}\cap\st{t\ge-c};\;\field_{\st{t> -c}}\tens F).
\end{align*}

\smallskip\noindent(ii)
Consider the natural maps
$$
\xymatrix{
N\times\bR \ar[r]^{i_\R} \ar[d]^{\pi_N} & M\times \bR \ar[d]^{\pi_M} \\
N\ar[r]^i&M
}$$
and set, for $S=M,N$,
\[
\field_{S\times\st{t>\ast}} \defeq \indlim[c\to+\infty]\field_{S\times\st{t>-c}} \in \BDC(\ifield_{S\times\bR}).
\]
Noticing that $\Eopb i K \simeq \field_N^\enh\ctens\quot\opb{i_\R}F$, by \cite[Proposition 6.6.5]{KS16D} one has
\begin{align*}
\sh(\Eopb i K)
&\simeq \alpha_N\roim{{\pi_N}}\bl\field_{N\times\st{t>\ast}}\tens \opb{i_\R}F\br \\
&\simeq \alpha_N\roim{{\pi_N}}\opb{i_\R}\bl\field_{M\times\st{t>\ast}}\tens F\br. 
\end{align*}
Hence
\begin{align*}
\rsect(T;\sh(\Eopb i K))
&\simeq \ilim[V]\rsect(V;\sh(\Eopb i K)\br \\
&\simeq\ilim[c,V]\rsect(V;
\roim{{\pi_N}}\opb{i_\R}(\field_{M\times\st{t>-c}}\tens F)\br \\
&\simeq\ilim[c,V]\rsect\bl V\times\R;\opb{i_\R}(\field_{M\times\st{t>-c}}\tens F)\br \\
&\simeq\ilim[c,V]\rsect\bl V\times\st{t\ge-c};
\opb{i_\R}(\field_{M\times\st{t>-c}}\tens F)\br \\
&\simeq \ilim[c,V,W]\rsect\bl W;\; \field_{M\times\st{t>-c}}\tens F\br\\
\end{align*}
where $c\to+\infty$, $V$ runs over the system of open relatively compact subanalytic neighborhoods of $T$ in $N$, 
and $W=W_{c,V}$ runs over the system of open subanalytic subsets of 
$M\times\st{\ct\in\cR;\;+\infty\ge \ct\ge-c}$, such that $W\supset V\times\st{t\in\R;\,t\ge -c}$. 
Here, the last isomorphism follows from Corollary~\ref{cor:FKU}.

\smallskip\noindent(iii)
For $c\in \R$ consider the following inductive systems: $I_c$ is the set of tuples $(U,\delta,\varepsilon)$ as in (i); $J_c$ is the set of tuples $(V,W)$ as in (ii). We are left to show the cofinality of the functor $\phi\colon I_c \to J_c$, $(U,\delta,\varepsilon)\mapsto(U\cap N,U_{c,\delta,\varepsilon}\cap\st{t\ge c})$.

Given $(V,W)\in J_c$, we look for $(U,\delta,\varepsilon)\in I_c$ such that $U\cap N\subset V$ and $U_{c,\delta,\varepsilon}\cap\st{t\ge-c}\subset W$.
Let $U$ be a subanalytic relatively compact open neighborhood of $T$ in $M$ such that $\overline U\cap N \subset V$. 
With notations as in Lemma~\ref{lem:Loj}, set
$X=M\times\st{\ct\in \overline\R\mid \ct\ge -c}$, $W=W$, $T=\overline U\times
\st{\ct\in \overline\R\mid \ct\ge -c}$, $f(x,\overline t)=\varphi(x)$ and 
\[
g(x,\overline t)=
(\overline t +c+1)^{-1}.
\]
Note that $g(x,+\infty)=0$.
Since \eqref{eq:Lojhyp} is satisfied,
Lemma~\ref{lem:Loj}~(ii) provides $C>0$ and $n\in\Z_{>0}$ such that
\[
\{(x,\overline t)\in\overline U\times\overline\R\semicolon 
\ct\ge-c,\;C g(x,\overline t)^n>|\varphi(x)|  \} \subset W.
\]
Then
\[
\{(x,t)\in U\times\R\semicolon \;t\ge -c,\;
C (t +c+1)^{-n}>|\varphi(x)|  \} \subset W.
\]
One concludes by noticing that the set on the left hand side contains $U_{c,\delta,\varepsilon}\cap\st{t\ge-c}$ for $\delta=C^{1/n}$ and $\varepsilon=1/n$.
\end{proof}

\section{Specialization and microlocalization}\label{se:numu}

Using results from the previous section, we establish here a germ formula for the natural enhancement of Sato's specialization and microlocalization functors, as introduced in \cite{DK19nu}. 

\subsection{Real oriented blow-up transforms}\label{sse:blow}
Let $M$ be a real analytic manifold and $N\subset M$ a closed submanifold.
Denote by $S_NM$ the sphere normal bundle.
Consider the real oriented blow-up $\RB M$ of $M$ with center $N$, which enters the commutative diagram with cartesian square
\begin{equation}\label{eq:blowMN}
\xymatrix@R=3ex{
S_NM \ar@{^(->}[r]^-{i} \ar[d]_\sigma & \RB M \ar[d]_{p}
& \inb{(M\setminus N)} \ar@{_(->}[dl]^-{j_N} \ar@{_(->}[l]_-{j} \\
N \ar@{^(->}[r]_-{i_N} \ar@{}[ur]|-\square & M.
}
\end{equation}
Recall the blow-up transform of  \cite[\S4.4]{DK19nu}
\[
\Enu^\rb\colon\BEC M \to \BEC{S_NM},\quad
K\mapsto \Eopb i\Eoim j\Eopb j_N K.
\]

A \emph{sectorial neighborhood} of $\theta\in S_NM$ is an open subset $U\subset M\setminus N$ such that $S_N M\cup  j (U)$ is a neighborhood of $\theta$ in $\RB M$. We write $U \dotowns \theta$ to indicate that $U$ is a sectorial neighborhood of $\theta$.
We say that $U\subset M\setminus N$ is a sectorial neighborhood of $Z\subset S_NM$, and we write $U\dotsupset Z$, if $U$ is a sectorial neighborhood of each $\theta\in Z$.

\begin{lemma}\label{lem:psiphigerm}
Let $\varphi\colon M\to \R$ be a subanalytic continuous map such that
$N = \varphi^{-1}(0)$.
Let $K\in\BECwRc M$.
For $\theta_0\in S_NM$, one has
\[
\sh\bl\Enu^\rb_N(K)\br_{\theta_0} \simeq
\ilim[\delta,\varepsilon, U]
\FHom( \exx^{\range0{-\delta|\varphi(x)|^{-\varepsilon}}}_{U|M} , K),
\]
where $\delta,\varepsilon\to 0+$ and $U\dotowns \theta_0$.
More generally, if $Z\subset S_NM$ is a \emph{closed} subset one has
\[
\rsect\bl Z;\sh(\Enu^\rb_N(K))\br \simeq
\ilim[\delta,\varepsilon,U]
\FHom( \exx^{\range0{-\delta|\varphi(x)|^{-\varepsilon}}}_{U|M} , K)
\]
where $\delta,\varepsilon\to 0+$ and $U\dotsupset Z$.
\end{lemma}

\begin{proof}
Let us prove the last statement.

Note that in $\RB M$ one has $S_NM = (\varphi\circ p)^{-1}(0)$. Hence, by Proposition~\ref{pro:sgerm},
\[
\rsect\bl Z;\sh(\Enu^\rb_N(K))\br
\simeq \ilim[\delta,\varepsilon,\widetilde U]\FHom(\exx^{\range0{-\delta|\varphi(p(\tilde x))|^{-\varepsilon}}}_{\widetilde U|\RB M},\Eoim j\Eopb j_N K),
\]
where $\widetilde U\subset\RB M$ runs over the neighborhoods of $i(Z)$.
Then
\begin{align*}
\rsect\bl Z;\sh(\Enu^\rb_N(K))\br
&\simeq \ilim[\delta,\varepsilon,\widetilde U]\FHom(\Eeeim{{j_N}}\Eopb j\exx^{\range0{-\delta|\varphi(p(\tilde x))|^{-\varepsilon}}}_{\widetilde U|\RB M}, K) \\
&\simeq \ilim[\delta,\varepsilon,\widetilde U]\FHom(\Eeeim{{j_N}}\exx^{\range0{-\delta|\varphi(x)|^{-\varepsilon}}}_{\opb j(\widetilde U)|\inb{(M\setminus N)}}, K) \\
&\simeq \ilim[\delta,\varepsilon,\widetilde U]\FHom(\exx^{\range0{-\delta|\varphi(x)|^{-\varepsilon}}}_{j_N(\opb j(\widetilde U))|M},K).
\end{align*}
One concludes by noticing that $U\dotsupset Z$ if and only if $U=j_N(\opb j(\widetilde U))$ for some neighborhood $\widetilde U$ of $i(Z)$  in $\RB M$.
\end{proof}

\subsection{Sheafification on vector bundles}
Recall from \cite[\S2.2]{DK19nu} that any morphism $p\colon M\to\bbS$, from a good space to a bordered space, admits a unique bordered compactification $p_\infty\colon\bM\to\bbS$ such that $(\bM){}^\circ=M$ and $p_\infty$ is semiproper.

Let $\tau\colon V\to N$ be a vector bundle.
Denote by $\inb V$ its bordered compactification, and by $o\colon N\to V$ the zero section.

The natural action of $\R_{>0}$ on $V$ extends to an action of the bordered group\footnote{a group object in the category of bordered spaces} $\bGmp \seteq (\R_{>0},\overline\R)$ on $\inb V$.
Denote by $\BECcon{\inb V}$ the category of conic enhanced ind-sheaves on $\inb V$.

\begin{lemma}\label{lem:oshsho}
For $K\in\BECcon{\inb V}$, one has
\[
\opb o \sh (K) \simeq \sh(\Eopb o K), \quad
\epb o \sh (K) \simeq \sh(\Eepb o K).
\]
\end{lemma}

\begin{proof}
We shall prove only the first isomorphism since the proof of the second
is similar.

With the identification $N\simeq o(N)\subset V$, set $\bdot V=V\setminus N$.
Consider the commutative diagram, associated with the real oriented blow-up of $V$ with center $N$.
\[
\xymatrix@R=3ex@C=8ex{
S_NV \ar[d] \ar@{^(->}[r] & \inb{(\RB[N]V)} \ar[d]^p \ar[r]^{\tilde\gamma }& S_NV \ar[dr]^\sigma \\
N \ar@{}[ur]|\square \ar@{^(->}[r]^o & \inb V \ar@/_1.0pc/[rr]_\tau & \inb{(\bdot V)} \ar@{_(->}[l]_j \ar[u]^\gamma \ar[r]^{\bdot\tau} \ar[ul]_{\tilde\jmath} & N 
}
\]
Consider the distinguished triangle
\[
\Eeeim j\Eopb j K \to K \to \Eoim o \Eopb o K \to[+1].
\]
One has
\[
\opb o\sh(\Eoim o \Eopb o K) \underset{(*)}\simeq
\opb o\oim o\sh(\Eopb o K) \simeq
\sh(\Eopb o K).
\]
where $(*)$ holds since $o$ is proper.
Hence, we can assume 
\[
K\simeq\Eeeim j\Eopb j K 
\]
and, since $\Eopb o K\simeq 0$, we have to show
\[
\opb o\sh(K)\simeq 0.
\]
 Recall that $\Eopb j K \simeq\Eopb \gamma K^\sph$ for $K^\sph\defeq \Eoim\gamma\Eopb j K$.
Then one  has
\begin{align*}
K
&\simeq\Eeeim j\Eopb \gamma K^\sph \\
&\simeq\Eoim p\Eeeim{\tilde\jmath}\Eopb{\tilde\jmath}\Eopb{\tilde\gamma} K^\sph \\
&\simeq\Eoim p\bl\field_{\RB[N]V\setminus S_NV}\tens\Eopb{\tilde\gamma} K^\sph\br.
\end{align*}
Thus, recalling that $\opb o\sh(K)\simeq\roim\tau\sh(K)$ since $\sh(K)$ is conic,
\begin{align*}
\opb o\sh(K)
&\simeq\roim\tau\sh\bl\Eoim p(\field_{\RB[N]V\setminus S_NV}\tens\Eopb{\tilde\gamma} K^\sph)\br \\
&\underset{(*)}\simeq\roim\tau\roim p\sh\bl\field_{\RB[N]V\setminus S_NV}\tens\Eopb{\tilde\gamma} K^\sph\br \\
&\simeq\roim\sigma\roim{\tilde\gamma}\sh\bl\field_{\RB[N]V\setminus S_NV}\tens\Eopb{\tilde\gamma} K^\sph\br ,
\end{align*}
where $(*)$ holds since $p$ is proper.
It is then enough to show
\[
\roim{\tilde\gamma}\sh(\field_{\RB[N]V\setminus S_NV}\tens\Eopb{\tilde\gamma} K^\sph) \simeq 0.
\]
Since $\tilde\gamma$ is borderly submersive and
$\epb{\tilde\gamma}\field_{S_NV}\simeq \field_{\RB[N]V\setminus S_NV}$,
one has by \eqref{eq:tsepb}
$$\field_{\RB[N]V\setminus S_NV}\tens\Eopb{\tilde\gamma} K^\sph
\simeq\Eepb{\tilde\gamma} K^\sph.$$
Hence one obtain
\eqn
\roim{\tilde\gamma}\sh(\field_{\RB[N]V\setminus S_NV}\tens\Eopb{\tilde\gamma} K^\sph) 
&&\simeq
\roim{\tilde\gamma}\sh(\Eepb{\tilde\gamma} K^\sph)\\ 
&&\underset{(*)}{\simeq}
\roim{\tilde\gamma}\epb{\tilde\gamma} \sh(K^\sph) \\
&&{\simeq}
\roim{\tilde\gamma}\rhom\bl\field_{\RB[N]V},\;\epb{\tilde\gamma} \sh(K^\sph)\br \\
&&{\simeq}
\rhom\bl\reim{\tilde\gamma}\field_{\RB[N]V},\;\sh(K^\sph)\br.
\eneqn
where $(*)$ follows from Lemma~\ref{lem:shop}~(i).
Then the desired result
follows from
$\reim{\tilde\gamma}\field_{\RB[N]V}\simeq0$.
\end{proof}

\subsection{Specialization and microlocalization}
Let us recall from \cite{DK19nu} the natural enhancement of Sato's specialization and microlocalization functors.

\smallskip
Let $M$ be a real analytic manifold and $N\subset M$ a closed submanifold.
Consider the normal and conormal bundles
\[
\xymatrix{T_NM \ar[r]^-\tau & N & \ar[l]_-\varpi T^*_NM,}
\]
and denote by $\inb{(T_NM)}$ and $\inb{(T^*_NM)}$ the bordered compactification of $\tau$ and $\varpi$, respectively.

Denote by $(p,s)\cl \ND{M}\to M \times \R$ the normal deformation of $M$ along $N$ (see \cite[\S4.1]{KS90}). Setting $\Omega\defeq\opb s(\R_{>0})$, one has morphisms
\begin{equation}\label{eq:nd}
\xymatrix@R=3ex{
\inb{(T_NM)} \ar@{^(->}[r]^-{i} & \inb{(\ND M)} & \inb\Omega \ar@{_(->}[l]_-{j} \ar[r]^-{p_\Omega} & M\,,
}
\end{equation}
where $\inb{(\ND M)}$ is the bordered compactification of $p$, and $p_\Omega=p|_\Omega$.
The enhanced Sato's specialization functor is defined by
\[
\Enu_N \colon \BEC M \to \BECcon{\inb{(T_NM)}}, \quad
K\mapsto \Eopb i \Eoim j\Eopb p_\Omega K .
\]

Sato's Fourier transform have natural enhancements (see e.g.\ \cite[\S5.2]{DK19nu})
\begin{align*}
(\cdot)^\wedge &\colon \BECp{\inb{(T_NM)}} \to \BECp{\inb{(T^*_NM)}}, \\
\lap(\cdot) &\colon \BECp{\inb{(T_NM)}} \to \BECp{\inb{(T^*_NM)}},
\end{align*}
and we denote by $(\cdot)^\vee$ and $\lapa(\cdot)$ their respective quasi-inverses.
 Recall that $(\cdot)^\wedge$ and $(\cdot)^\vee$ take values in conic objects, and that $\lap(\cdot)$ and $\lapa(\cdot)$ send conic objects to conic objects. 

Finally, Sato's  microlocalization functor have a natural enhancement
\[
\Emu_N \colon \BECp M \to \BECp{\inb{(T^*_NM)}}\cap \BECcon{\inb{(T^*_NM)}},
\]
defined by $\Emu_N(K) \defeq \lap\Enu_N(K) \simeq \Enu_N(K)^\wedge$.

Consider the natural morphisms
\[
\xymatrix{S_NM & \ar[l]_\gamma \inb{(\bdot T_NM)} \ar[r]^u & \inb{(T_NM)} & N \ar[l]_-{o}},
\]
where $\bdot T_NM$ is the complement of the zero-section, and $o$ is the embedding of the zero-section.
Recall that one has 
\[
\Eopb\gamma \circ \Enu_N^\rb \simeq \Eopb u \circ \Enu_N.
\]

Recall from \cite[\S4.1]{KS90} that the normal cone $C_N(S)\subset T_NM$ to $S\subset M$ along $N$ is defined by $C_N(S) \defeq T_NM\cap\overline{\opb{p_\Omega}(S)}$,
where $\overline{(\cdot)}$ denotes the closure in $\ND M$.

\begin{lemma}\label{lem:numugerm}
Let $\varphi\colon M\to \R$ be a continuous subanalytic function such that
$N = \varphi^{-1}(0)$.
For $v_0\in T_NM$, $\xi_0\in T^*_NM$, and $K\in\BECwRc M$, one has
\begin{align}
\tag{i}
\sh\bl\Enu_N(K)\br_{v_0} &\simeq
\ilim[\delta,\varepsilon,U]
\FHom( \exx^{\range0{-\delta|\varphi(x)|^{-\varepsilon}}}_{U|M} , K), \\
\tag{ii}
\sh\bl\Emu_N(K)\br_{\xi_0} &\simeq
\ilim[\delta,\varepsilon, W, Z]
\FHom(\exx^{\range0{-\delta|\varphi(x)|^{-\varepsilon}}}_{W\cap Z|M} , K),
\end{align}
where $\delta,\varepsilon\to0+$, $U$ runs over the open subsets of $M$ such that $v_0\notin C_N(M\setminus U)$, $W$ runs over the open neighborhoods of $\varpi(\xi_0)$ in $M$, and
$Z$ runs over the closed subsets of $M$ such that 
\[
C_N(Z)_{\varpi(\xi_0)}\subset\{v\in(T_NM)_{\varpi(\xi_0)}\semicolon \langle v,\xi_0\rangle>0\}\cup\{0\}.
\]
\end{lemma}

\begin{proof}
(i-a) Assume that $v_0\in\bdot T_NM$, and set $\theta_0=\gamma(v_0)$. Then, one has
\begin{align*}
\sh\bl\Enu_N(K)\br_{v_0}
&\underset{(*)}\simeq \sh\bl\Eopb u\Enu_N(K)\br_{v_0} \\
&\simeq \sh\bl\Eopb \gamma\Enu^\rb_N(K)\br_{v_0} \\
&\underset{(**)}\simeq \sh\bl\Enu^\rb_N(K)\br_{\theta_0},
\end{align*}
where $(*)$ and $(**)$ follow from Lemma~\ref{lem:shop}~(i).
Then, the statement follows from Lemma~\ref{lem:psiphigerm}~(i), by noticing
that $U\dotowns \theta_0$ if and only if $v_0\notin C_N(M\setminus U)$.

\medskip\noindent (i-b)
Assume that $v_0=o(y_0)$ for $y_0\in N$, where $o\colon N\to T_NM$ is the embedding of the zero section. Then, Lemma~\ref{lem:oshsho} gives
\begin{align*}
\sh\bl\Enu_N(K)\br_{o(y_0)}
&\simeq \bl\opb o \sh(\Enu_N(K))\br_{y_0} \\
&\simeq \bl \sh(\Eopb o\Enu_N(K))\br_{y_0} \\
&\underset{(*)}\simeq \bl \sh(\Eopb {i_N}K)\br_{y_0},
\end{align*}
where $(*)$ follows from \cite[Lemma~4.8~(i)]{DK19nu}.
Then the statement follows from Proposition~\ref{pro:sgerm}.

\medskip\noindent (ii) 
For $F\in\BDC_{\Gmp}(\field_{T^*_NM})$ one has
\begin{align*}
\FHom(\epsilon(F),\Emu_N(K))
&= \FHom(\epsilon(F),\lap\Enu_N(K)) \\
&\simeq \FHom(\lapa \epsilon(F),\Enu_N(K)) \\
&\simeq \FHom(\epsilon(F^\vee),\Enu_N(K)) .
\end{align*}
Hence
\begin{align*}
\sh\bl\Emu_N(K)\br_{\xi_0} 
&\simeq \ilim[V\owns\xi_0] \FHom(\epsilon(\field_V),\Emu_N(K)) \\
&\simeq \ilim[V\owns\xi_0] \FHom(\epsilon(\field_V^\vee),\Enu_N(K)) \\
&\underset{(*)}\simeq \ilim[V\owns\xi_0] \FHom(\epsilon(\field_{V^\circ}),\Enu_N(K)) ,
\end{align*}
where $V$ runs over the conic open neighborhoods of $\xi_0$ in $T^*_NM$,
 and $V^\circ \defeq \{v\in T_NM \semicolon \langle v,\xi \rangle \geq 0, \ \forall
\xi\in V \}$
denotes the polar cone.
Here $(*)$ follows by noticing that
$\xi_0$ has a fundamental system of open conic neighborhoods $V\subset T^*_NM$ such that $\tau|_V$ has convex fibers.

We are left to compute $\ilim[V\owns\xi_0] \FHom(e(F),\Enu_N(K))$ for
$F=\field_{V^\circ}$.
For this, setting $ W=\tau(V)$, and considering the distinguished triangle
\[
\field_{\tau^{-1}( W)\setminus V^\circ} \to
\field_{\tau^{-1}( W)} \to
\field_{V^\circ} \to[+1],
\]
we will instead compute the cases where
$F=\field_{\tau^{-1}( W)}$ or $F=\field_{\tau^{-1}( W)\setminus V^\circ}$.

On one hand, one has
\begin{align*}
\FHom(\epsilon(\field_{\tau^{-1}( W)}),\Enu_N(K))
&\simeq \FHom(\Eopb\tau \epsilon(\field_{ W}),\Enu_N(K)) \\
&\simeq \FHom(\epsilon(\field_{ W}),\Eoim\tau \Enu_N(K)) \\
&\simeq \FHom(\epsilon(\field_{ W}),\Eopb i K).
\end{align*}
Thus, noticing that $W=\tau(V)$ is a system of neighborhoods of $\varpi(\xi_0)$,
\begin{align*}
\ilim[V\owns\xi_0] \FHom(\epsilon(\field_{\tau^{-1}( W)}),\Enu_N(K))
&\simeq \ilim[ W\owns \varpi(\xi_0)] \FHom(\epsilon(\field_ W),\Eopb i K) \\
&\underset{(*)}\simeq \ilim[\delta,\varepsilon,W]
\FHom( \exx^{\range0{-\delta|\varphi(x)|^{-\varepsilon}}}_{ W|M} , K),
\end{align*}
where $(*)$ follows from Proposition~\ref{pro:sgerm}.

On the other hand,
setting $\widetilde V = \gamma(\tau^{-1}( W)\setminus V^\circ)\subset S_NM$, one has
$\field_{\tau^{-1}( W)\setminus V^\circ} \simeq \reim u\opb\gamma\field_{\widetilde V}$.
Hence
\begin{align*}
\FHom(\epsilon(\field_{\tau^{-1}( W)\setminus V^\circ}),\Enu_N(K))
&\simeq \FHom(\Eeeim u\Eopb\gamma \epsilon(\field_{\widetilde V}),\Enu_N(K)) \\
&\simeq \FHom(\epsilon(\field_{\widetilde V}),\Eoim\gamma\Eopb u\Enu_N(K)) \\
&\simeq \FHom(\epsilon(\field_{\widetilde V}),\Enu^\rb_N(K)).
\end{align*}
Note that when $V$ runs over the neighborhoods of $\xi_0$, $\widetilde V$ runs over the neighborhoods of  $Z=\gamma(\{\xi_0\}^\circ)$. Thus
\begin{align*}
\ilim[V\owns\xi_0] \FHom(\epsilon(\field_{\tau^{-1}( W)\setminus V^\circ}),\Enu_N(K))
&\simeq \ilim[V\owns\xi_0] \FHom(\epsilon(\field_{\widetilde V}),\Enu^\rb_N(K)) \\
&\simeq \ilim[V\owns\xi_0] \RHom(\field_{\widetilde V},\sh\bl\Enu^\rb_N(K)\br) \\
&\simeq \rsect\bl Z;\sh(\Enu^\rb_N(K))\br \\
&\underset{(*)}\simeq \ilim[\delta,\varepsilon,U]
\FHom( \exx^{\range0{-\delta|\varphi(x)|^{-\varepsilon}}}_{U|M} , K),
\end{align*}
where $\delta,\varepsilon\to0+$, and $U\dotsupset Z$.
Here, $(*)$ follows from Lemma~\ref{lem:psiphigerm}~(iii).
\end{proof}

\appendix

\section{Complements on enhanced ind-sheaves}\label{se:compl}

We provide here some complementary results on (enhanced ind-)sheaves that we need in this paper. 
\medskip

\Prop
Let $\bbM$ be a subanalytic bordered space, and $\bbN$ a bordered space.
Then, for any $F\in\BDC_{\Rc}(\field_\bbM)$ and $K\in\BDC(\ifield_\bbN)$
we have
\eq \label{eq:DFK}
\dual_\bbM F\etens K\simeq \rihom
(\opb p F,\epb q K).
\eneq
Here, $p\cl \bbM\times\bbN\to\bbM$ and $q\cl\bbM\times\bbN\to\bbN$
are the projections.
\enprop
\begin{proof}
By \cite[Proposition 2.3.4]{DK16}, one has
\[
\dual_\cbM\reim{{j_\bbM}} F\etens \reim{{j_\bbN}} K\simeq \rihom
(\cc p{}^{\,-1}\reim{{j_\bbM}} F,\cc q{}^{\,!}\reim{{j_\bbN}} K),
\]
where $\cc p$ and $\cc q$ are the projections from $\cbM\times\cbN$, and $j_\bbM\colon\bbM\to\cbM$ is the natural morphism.

Applying $\opb j_{\bbM\times\bbN}$, \eqref{eq:DFK} follows.
\end{proof}

\Prop\label{prop:extd}
Let $\bbM$, $\bbN$, $F$, $K$ be as in the preceding proposition.
Let $f\cl\bbN\to\bbS$ be a morphism of bordered spaces,
and let $f'=\id_ \bbM\times f\cl \bbM\times\bbN\to\bbM\times\bbS$.
Then, we have
$$\roim{f'}(F\etens K)\simeq F\etens\roim f K.$$
\enprop
\Proof
Let $p_\bbN\cl\bbM\times\bbN\to\bbM$ and $q_\bbN\cl\bbM\times\bbN\to\bbN$ 
be the projections.
We define similarly $p_\bbS$ and $q_\bbS$.
Then, the preceding proposition implies
\eqn
\roim{f'}\; (F\etens K)
&&\simeq \roim{f'}\;\rihom \bl{\opb p_\bbN}\dual_\bbM F, \epb{q_\bbN} K\br\\
&&\simeq \roim{f'}\;\rihom \bl\opb{f'}{\opb p_\bbS}\dual_\bbM F, \epb{q_\bbN} K\br\\
&&\simeq \rihom \bl{\opb p_\bbS}\dual_\bbM F, \roim{f'}\epb{q_\bbN} K\br\\
&&\simeq \rihom \bl{\opb p_\bbS}\dual_\bbM F, \epb{q_\bbS}\roim{f}K\br\\
&&\simeq F\etens \roim{f}K.
\eneqn
\QED

\begin{lemma}\label{lem:carsq} Let us consider a  commutative  square of  bordered  spaces
$$\xymatrix@C=7ex{ \bbM'\ar[r]^{g'}\ar[d]^-{f'}& \bbM\ar[d]^-f\\
 \bbN'\ar[r]^g& \bbN.}
$$
For any $F\in\BDC(\ifield_\bbM)$, one has a canonical morphism
in $\BDC(\ifield_{ \bbN'})$
$$\opb{g}\roim f F\to\roim{f'}\opb{g'}F.$$
If  the square is cartesian and  $g$ is  borderly  submersive,
then the above morphism is an isomorphism.
\end{lemma}

\Proof
The morphism is induced by adjunction from
$$\roim f F\to\roim{f}\roim {g'}\opb{g'}F\isoto\roim{g}\roim {f'}\opb{g'}F.$$

 Assume that the square is cartesian and  $g$ is  borderly  submersive. Then we may assume  that 
$ \bbN'=S\times  \bbN$ and $ \bbM'=S\times  \bbM$ for a subanalytic space $S$,  and that $g$ and $g'$ are the second projections.
Hence the assertion follows from 
$\roim{f'}g^{\prime\,-1}F\simeq\roim{f'}(\cor_S\etens F)\simeq\cor_S\etens \roim f F\simeq\opb{g}\roim f F$,
which is a consequence of Proposition~\ref{prop:extd}.
\QED

\begin{lemma}\label{lem:A3}
 For $f\cl\bbM\to\bbN$ a morphism of bordered spaces and $K\in\BEC\bbN$
there is a natural morphism $\opb{f_\R}(\RE K)\to\RE(\Eopb f K)$.
If $f$ is borderly submersive, then the previous morphism is an isomorphism.
\end{lemma}
\Proof
 The morphism in the statement follows by adjunction from the isomorphism 
$\quot_\bbM(\opb{f_\R}\RE K)\simeq\Eopb f K$.
If $f$ is borderly submersive, we have
\eqn
\roim{{\pi_\bbM}}\opb{f_\R}\RE K
&&\underset{(*)}\simeq\opb{f}\roim{{\pi_\bbN}}\RE K
\simeq0,
\eneqn
where $(*)$ follows from Proposition~\ref{prop:extd}.
Hence, the fact that the morphism in the statement is an isomorphism follows from \cite[Proposition~4.4.4 (ii-b)]{DK16}.
\QED

\section{Complements on weak constructibility}\label{se:wc}

In this appendix we obtain a formula for the sections, on a locally closed subanalytic subset, of a weakly constructible sheaf. This result might be of independent interest. 
\medskip

\subsection{Lojasiewicz's inequalities}

Let $ M$ be a subanalytic space.

\begin{lemma}\label{lem:Loj}
Let  $T\subset  M$ be a compact subanalytic subset, and
let $f,g\colon  M\to \R$ be continuous subanalytic functions.
\begin{itemize}
\item[(i)]
Assume that $T\cap\opb f(0)\subset\opb g(0)$. Then there exist $\varepsilon>0$ and $n\in\Z_{>0}$ such that
\[
\varepsilon|g(x)|^n \leq |f(x)| \quad\text{for }x\in T.
\]
\item[(ii)]
Let $W\subset  M$ be an open subanalytic subset, 
and assume that
\begin{equation}\label{eq:Lojhyp}
\{x\in T\semicolon g(x)>0,\ f(x)=0  \} \subset W.
\end{equation}
Then there exist $\varepsilon>0$ and $n\in\Z_{>0}$ such that
\[
\{x\in T\semicolon g(x)>0,\ \varepsilon g(x)^n>|f(x)|  \} \subset W.
\]
\end{itemize}
\end{lemma}

\begin{proof}
Consider the subanalytic map $(f,g)\colon  M\to\R^2_{(t,u)}$.

\smallskip\noindent(i)
The set $Z=(f,g)(T)$ is a compact subanalytic subset of $\R^2$, and we have
\[
Z\cap\st{(t,u)\semicolon t=0}\subset\st{(t,u)\semicolon u=0}.
\]
Hence, there exist $\varepsilon>0$ and $n\in\Z_{>0}$ such that
\[
Z \subset \st{(t,u)\in\R^2\semicolon \varepsilon|u|^n\leq |t|}.
\]
This gives the statement.

\smallskip\noindent(ii)
Let $T'=T\cap\opb g(\R_{\geq 0})\setminus W$. Since $T'\cap\opb f(0)\subset\opb g(0)$,  (i) gives
\[
T'\subset\st{x\in  M\semicolon \varepsilon|g(x)|^n \leq |f(x)|},
\]
which implies the desired result.
\end{proof}

\begin{theorem}\label{th:subanalytic ext}
Let $M$ be a subanalytic space, and $F\in\BDC_{\wRc}(\field_M)$.
Then, for any locally closed subanalytic subset $Z$ of $M$, and any open subanalytic subset $W$ of $M$ such that $Z\subset W$, there exists
$U\subset W$ open subanalytic in $M$, such that
$Z$ is a closed subset of $U$ and
\[
\rsect(U;F)\isoto \rsect(Z;F).
\]
\end{theorem}

The proof is given in \S~\ref{subsec:proof}
after the preparation of the next subsection.

\begin{corollary}\label{cor:FKU}
Let $\bbM$ be a subanalytic bordered space,
$Z$ a locally closed subanalytic subset of $\bbM$, and let $F\in\BDC_{\wRc}(\field_\bbM)$.
Then, there is an isomorphism
\[
\rsect(Z;F) \isofrom \indlim[U]\rsect(U;F),
\]
where $U$ runs over the open subanalytic subsets of $\bbM$ such that $Z\subset U$.
\end{corollary}

\subsection {Barycentric decomposition}
We will use here the language of simplicial complexes, for which we refer to \cite[\S8.1]{KS90}.

Let $\Sigma=(S,\Delta)$ be a simplicial complex, with $S$ the set of vertices, and $\Delta$ the set of simplexes (i.e., finite subsets of $S$).
Recall that one sets $|\Sigma|\defeq\Union\nolimits_{\sigma\in\Delta}|\sigma|$, where
\[
|\sigma|\defeq\st{x\in\R^S\semicolon \sum_p x(p)=1,\;x(p)=0\text{ for }p\notin\sigma,\; x(p)>0\text{ for }p\in\sigma }.
\] 
Here, $\R^S$ denote the set of maps $S\to\R$ equipped with the product topology.

For a subset $Z$ of $|\Sigma|$,
we set $$\Delta_Z\seteq\set{\sigma\in \Delta}{|\sigma|\subset Z}.$$

A subset $Z$ of $|\Sigma|$ is called $\Sigma$-constructible if
$Z$ is a union of simplexes.

\begin{lemma}\label{lem:Zloccl}
Let $Z$ be a $\Sigma$-constructible subset of $|\Sigma|$.
\begin{itemize}
\item[(i)] the following conditions are equivalent.
\begin{itemize}
\item[(a)] $Z$ is closed,
\item[(b)]
if $\tau, \sigma\in\Delta$ satisfy
$\sigma\in\Delta_Z$ and $\tau\subset \sigma$, then
$\tau\in \Delta_Z$.
\end{itemize}
\item[(ii)] the following conditions are equivalent.
\begin{itemize}
\item[(a)] $Z$ is open
\item[(b)]
if $\tau, \sigma\in\Delta$ satisfy
$\sigma\in\Delta_Z$ and $\sigma\subset \tau$, then
$\tau\in \Delta_Z$.
\end{itemize}
\item[(iii)] the following conditions are equivalent.
\begin{itemize}
\item[(a)] $Z$ is locally closed,
\item[(b)]
if $\sigma_1, \sigma_2,\sigma_3\in\Delta$ satisfy
$\sigma_1,\sigma_3\in\Delta_Z$ and $\sigma_1\subset \sigma_2\subset \sigma_3$, 
then
$\sigma_2\in \Delta_Z$.
\end{itemize}
\end{itemize}
\end{lemma}

\begin{proof}
(i) follows from $\ol{|\sigma|}=\bigcup_{\tau\in\Delta,\tau\subset\sigma}|\tau|$.
(ii) and (iii) follow from (i).
\end{proof}

For $\sigma\in\Delta$, we set
$$U(\sigma)=\bigcup_{\sigma\subset \tau\in\Delta}|\tau| = \st{x\in|\Sigma|\semicolon x(s)>0\text{ for any } s\in\sigma}.$$
It is the smallest open $\Sigma$-constructible subset containing $|\sigma|$.

Let us denote by $\BDC_{w\text-\Sigma-c}(\field_{|\Sigma|})$ the full subcategory of $\BDC(\field_{|\Sigma|})$ whose objects are weakly $|\Sigma|$-constructible.
By \cite[Proposition 8.1.4]{KS90}, we have

\begin{lemma} \label{lem:Usigma}
Let $F\in\BDC_{w\text-\Sigma-c}(\field_{|\Sigma|})$ and $\sigma\in\Delta$. Then, one has
$$\rsect(U(\sigma);F)\isoto\rsect(|\sigma|;F).$$
\end{lemma}

\bigskip
Let $\Bd(\Sigma)=(S_{\Bd(\Sigma)}, \Delta_{\Bd(\Sigma)})$ be the barycentric
decomposition of $\Sigma$ defined as follows:
\eqn
S_{\Bd(\Sigma)}&&=\Delta,\\
\Delta_{\Bd(\Sigma)}&&=\set{\tilde\sigma}
{\text{$\tilde\sigma$ is a finite totally ordered subset of $\Delta$}}.
\eneqn
Here, $\Delta_{\Bd(\Sigma)}$ is ordered by the inclusion relation.
Then there is a homeomorphism
$f\cl|\Bd(\Sigma)|\isoto |\Sigma|$
defined as follows.
For $\sigma\in\Delta=S_{\Bd(\Sigma)}$, let $e_\sigma\in|\Sigma|$ be given by
$$e_\sigma(s)=
\begin{cases}
\dfrac{1}{\sharp\sigma}&\text{if $s\in\sigma$,}\\[1ex]
0&\text{otherwise.}
\end{cases}
$$
Then, we define
\[
f(x)=\sum_{\sigma\in S_{\Bd(\Sigma)}}x(\sigma)e_{\sigma}
\quad\text{for any $x\in |\Bd(\Sigma)|\subset\R^{S_{\Bd(\Sigma)}}$.}
\]
That is, $f(x)\in\R^S$ is given by
\[
\bl f(x)\br(s)=\sum_{\sigma\owns s,\ \sigma\in S_{\Bd(\Sigma)}}\dfrac{x(\sigma)}{\sharp\sigma} 
\quad\text{for any $s\in S$.}
\]
Note that we have
\eq
f(|\tilde\sigma|)\subset |\max(\tilde\sigma)|\quad\text{for any $\tilde\sigma\in \Delta_{\Bd(\Sigma)}$,}\label{eq:bary}
\eneq
where $\max(\tilde\sigma)\in\Delta$ is the largest member of $\tilde\sigma\subset\Delta$.
Conversely, for $y\in|\Sigma|$ one has
$$y\in f(|\tilde\sigma|),$$
where $\tilde\sigma\in\Delta_{\Bd(\Sigma)}$ is given by
$$\tilde\sigma\defeq\set{\sigma\in\Delta}{\text{
$\sigma=\set{s\in S}{y(s)\ge a}$
for some $a\in\R_{>0}$}}.$$

\Lemma\label{sublemma}
Let $Z\subset|\Sigma|$ be a locally closed $\Sigma$-constructible subset.
Then for any
$\tilde\sigma_1,\tilde\sigma_2\in \Delta_{\Bd(\Sigma)}$ such that
$\tilde\sigma_1\cup\tilde\sigma_2\in\Delta_{\Bd(\Sigma)}$ and
 $f(|\tilde\sigma_1|),f(|\tilde\sigma_2|)\subset Z$,
we have
$f(|\tilde\sigma_1\cup\tilde\sigma_2|)\subset Z$.
\enlemma

\Proof
Set $\tilde\tau=\tilde\sigma_1\cup\tilde\sigma_2$.
We have
$|\max(\tilde\sigma_1)|,\,|\max(\tilde\sigma_2)|\subset Z$.
Then the desired result follows from the fact that
$\max(\tilde\tau)$ is equal to either
$\max(\tilde\sigma_1)$ or $\max(\tilde\sigma_2)$.
Hence $|\tilde\tau|\subset|\max(\tilde\tau)|\subset Z$.
\QED

\subsection{Proof of Theorem~\ref{th:subanalytic ext}}\label{subsec:proof}
\Lemma\label{lem:contr}
Let $\Sigma=(S,\Delta)$ be a  simplicial complex.
Let $Z\subset|\Sigma|$ be a $\Sigma$-constructible locally closed subset such that
\eq&&\ba{l}
\text{for any $\sigma_1,\sigma_2\in\Delta_Z$ such that
$\sigma_1\cup \sigma_2\in\Delta$,}\\
\text{one has $\sigma_1\cup\sigma_2\in\Delta_Z$.}\ea
\label{cond:Sigma}
\eneq
Set
$$U\seteq\bigcup_{\sigma\in\Delta_Z}U(\sigma).$$
Then, for $F\in\BDC_{w\text-\Sigma-c}(\field_{|\Sigma|})$ one has
$$\rsect(U;F)\isoto\rsect(Z;F).$$
\enlemma
\Proof
Let us remark that $U$ is an open subset and $Z$ is a closed subset of $U$.
Hence it is enough to how that
$$\rsect(U;F\tens\field_{U\setminus Z})\simeq0.$$
Thus, we reduce the problem to prove that
$\rsect(U;F)\simeq0$ under the condition 
that $F\in\BDC_{w\text-\Sigma-c}(\field_{|\Sigma|})$ satisfies $F\vert_Z\simeq0$.

Let us take the open covering 
$\mathfrak{U}\seteq\{U(\sigma)\}_{\sigma\in\Delta_Z}$ of $U$.
For $\sigma_1,\ldots,\sigma_\ell\in\Delta_Z$,
if $\bigcap_{1\le k\le \ell}U(\sigma_k)\not=\emptyset$, then
$\sigma\seteq\bigcup_{1\le k\le \ell}\sigma_k\in\Delta_Z$ 
by condition \eqref{cond:Sigma} and 
$\bigcap_{1\le k\le \ell}U(\sigma_k)=U(\sigma)$.

Hence, one has  by Lemma~\ref{lem:Usigma}
$$\rsect(\bigcap_{1\le k\le \ell}U(\sigma_k);F)
\isoto\rsect(|\sigma|;F)\simeq0.$$
Thus, we have
$\rsect(\bigcap_{1\le k\le \ell}U(\sigma_k);F)\simeq0$ for any
$\sigma_1,\ldots,\sigma_\ell\in\Delta_Z$.
We conclude that $\rsect(U;F)\simeq\rsect(\mathfrak{U};F)\simeq0$.
\QED

\Proof[{Proof of {\rm Theorem ~\ref{th:subanalytic ext}}}]
There exists a simplicial complex $\Sigma=(S,\Delta)$
and a subanalytic isomorphism $M\simeq|\Sigma|$
such that
$Z$ and $W$ are $\Sigma$-constructible and
$F$ is weakly $\Sigma$-constructible (after identifying $M$ and $|\Sigma|$).
Let $\tilde\Sigma=(\tilde S,\tilde\Delta)$ be the barycentric decomposition of $\Sigma$,
and identify $|\tilde\Sigma|$, $|\Sigma|$ and $M$.
Then $F$ is weakly $\tilde\Sigma$-constructible and
$Z$ and $W$ are $\tilde\Sigma$-constructible.
Set $U=\bigcup_{\tilde\sigma\in\tilde\Delta_Z}U(\tilde\sigma)$.
Then $U\subset W$ by Lemma~\ref{lem:Zloccl}.
Moreover, condition \eqref{cond:Sigma} is satisfied by Lemma~\ref{sublemma}. 
Hence, Lemma~\ref{lem:contr}
implies that
$\rsect(U;F)\to\rsect(Z;F)$ is an isomorphism.
\QED


\begin{thebibliography}{10}

\bibitem{DK16} A. D'Agnolo and M. Kashiwara,
\emph{Riemann-Hilbert correspondence for holonomic D-modules},
Publ. Math. Inst. Hautes \'Etudes Sci.
{\bf 123} (2016), no. 1, 69--197.

\bibitem{DK18} \bysame,
\emph{A microlocal approach to 
the enhanced Fourier-Sato transform in dimension one},
Adv. Math. {\bf 339} (2018), 1--59.

\bibitem{DK19} \bysame,
\emph{Enhanced perversities},
J. Reine Angew. Math. (Crelle's Journal) {\bf 751} (2019), 185--241.

\bibitem{DK19nu} \bysame,
\emph{Enhanced specialization and microlocalization},
{\tt arxiv:1908.01276} (2019), 29 pp.

\bibitem{GS14} S. Guillermou and P. Schapira,
\emph{Microlocal theory of sheaves and Tamarkin's non displaceability theorem},
in: Homological Mirror Symmetry
and Tropical Geometry, Lecture Notes of the Unione Matematica Italiana {\bf 15}, Springer, Berlin (2014), 43--85.

\bibitem{Kas84} M. Kashiwara,
\emph{The Riemann-Hilbert problem for holonomic systems},
Publ. Res. Inst. Math. Sci. {\bf 20} (1984), no. 2, 319--365.

\bibitem{Kas16} \bysame,
\emph{Riemann-Hilbert correspondence for irregular holonomic $\D$-modules},
 Jpn. J. Math. {\bf 11} (2016), no. 1, 113--149. 

\bibitem{KS90} M. Kashiwara and P. Schapira,
{\em Sheaves on manifolds},
Grundlehren der Mathematischen Wissenschaften \textbf{292}, Springer,
Berlin (1990), x+512 pp.

\bibitem{KS01} \bysame,
{\em Ind-sheaves},
Ast\'erisque \textbf{271} (2001), 136 pp.

\bibitem{KS16D} \bysame,
\emph{Regular and irregular holonomic D-modules},
London Mathematical Society Lecture Note Series \textbf{433}, Cambridge University Press, Cambridge (2016), vi+111 pp.

\bibitem{Tam08} D. Tamarkin,
\emph{Microlocal condition for non-displaceability},
in: Algebraic and Analytic Microlocal Analysis,
Springer Proc. in Math. \& Stat. {\bf 269} (2018), 99--223.

\end{thebibliography}
\end{document}